\documentclass[12pt]{article}

\usepackage[utf8]{inputenc}
\usepackage[T1]{fontenc}
\usepackage[english]{babel}
\usepackage{lmodern}
\usepackage[top=2cm, bottom=2.5cm, left=2cm, right=2cm]{geometry}
\usepackage{mathtools}
\usepackage{graphicx}
\usepackage{enumitem}
\usepackage[thmmarks, amsmath]{ntheorem}
\usepackage{hyperref}
\usepackage{amssymb}
\usepackage{pifont}
\usepackage{mathrsfs}
\usepackage{stmaryrd}
\usepackage{dsfont}
\usepackage{soul}
\usepackage{tabularx}
\usepackage{accents}
\usepackage[table]{xcolor}
\usepackage[framed, numbered]{matlab-prettifier}
\usepackage{changepage}
\usepackage{makecell}
\usepackage[skip=0.75\baselineskip plus 2pt]{parskip}
\usepackage{comment}
\usepackage[autostyle=true]{csquotes}
\usepackage[backend=biber, url=false, doi=false, isbn=false, style=alphabetic, sorting=ynt]{biblatex}
\usepackage{multicol}
\usepackage{pdfoverlay}
\usepackage[skip=0.5\baselineskip]{caption}
\usepackage{indentfirst}
\usepackage{float}
\usepackage{caption}
\usepackage{blindtext}
\usepackage{subcaption}
\usepackage{setspace}
\usepackage{import}
\usepackage{tabularray}
\usepackage[group-separator=\text{,}]{siunitx}
\usepackage{fancyhdr}
\usepackage{authblk}
\usepackage{needspace}

\theoremstyle{plain}
\theoremheaderfont{\upshape\bfseries}
\theorembodyfont{\itshape}
\theoremseparator{.}
\newtheorem{theorem}{Theorem}[section]
\newtheorem{proposition}[theorem]{Proposition}

\newtheorem{lemma}[theorem]{Lemma}
\newtheorem{definition}[theorem]{Definition}
\newtheorem{notation}[theorem]{Notation}

\newtheorem{condition}{Condition}

\theoremstyle{plain}
\theoremheaderfont{\upshape\bfseries}
\theorembodyfont{\itshape}
\theoremseparator{.}

\theoremstyle{nonumberplain}
\theoremheaderfont{\upshape\bfseries}
\theorembodyfont{\upshape}
\theoremseparator{.}
\newtheorem{remark}{Remark}[section]

\newtheorem{examples}{Examples}[section]

\theoremstyle{nonumberplain}
\theoremheaderfont{\itshape}
\theorembodyfont{\upshape}
\theoremsymbol{\ensuremath{\Box}}
\newtheorem{proof}{Proof}[section]

\numberwithin{equation}{section}

\DeclareMathOperator{\e}{e}

\allowdisplaybreaks

\setlist{nosep}

\newcommand{\HRule}{\rule{\linewidth}{0.5mm}}
\title{\makeatletter
\HRule \\[0.4cm]
{ \huge \bfseries Fourier Representations of Spectral Densities in Long-Memory Processes}\\[0.4cm] 
\HRule \\[1.5cm]}

\date{May 19, 2026}
\author{Valentin Vidril\thanks{Email: valentin.vidril@edu.escp.eu}}
\affil{
    CERESSEC, ESSEC Business School \& CY Cergy Paris University\\
    Econophysics Lab, Institut Louis Bachelier\thanks{28 Place de la Bourse, Palais Brongniart, 75002 Paris, France}
}

\begin{document}

\maketitle

\begin{abstract}
In this article, we aim to further clarify certain subtle aspects of processes that exhibit long memory in the second-order sense. We construct a long-memory stochastic sequence, in the sense that the series of absolute autocovariances diverges, whose spectral density has an almost everywhere unboundedly divergent Fourier series. This suggests that the Fourier series of the spectral density of a generic long-range dependent process, one for which nothing is known except that its autocovariances are not absolutely summable, should be handled with great care. On the other hand, it is known that if one assumes regularly varying behavior for the autocovariances or the spectral density, along with suitable conditions on the associated slowly varying function, then the Fourier series of the spectral density converges everywhere, except possibly at 0. The process we construct can easily be simulated, and we compare its empirical and theoretical autocovariances.
\end{abstract}

\newpage

\section{Introduction}

The concept of long memory has been studied for over 60 years, although the non-standard behaviour of long-range dependent processes had been observed for at least two centuries \cite{beran1994,beran2013}. For example, in 1951, Hurst discovered an empirical law during a study of the long-term storage capacity of reservoirs on the Nile \cite{hurst1951,hurst1956}. He examined the \textit{rescaled range} (also called the \textit{$R/S$ statistic}) of the Nile's water flow and found that it grows as a function of the number $n$ of observations, approximately as $n^{0.74}$. It has been shown by Feller that this empirical law is incompatible with the assumption of weak dependence \cite{feller1951}.

Since then, the literature has proposed a number of definitions to formalize this notion of dependence, with primary emphasis on long-range dependence as a second-order feature of a time series, that is, as a property related to the autocovariances or the spectral measure of the process.

In this approach, long memory characterizes the rate at which the autocovariances decay over time. At an intuitive level, a time series is said to exhibit long-range dependence if its autocovariances decay sufficiently slowly as the lag tends to infinity. To formalize this notion of slow second-order decay, several (not necessarily equivalent) definitions have been proposed in the literature. The weakest rigorous second-order definition of long memory requires only that the autocovariances decay slowly enough for the series of their absolute values to diverge. By contrast, short-memory processes have autocovariances that are absolutely summable and often decay exponentially fast.

To characterize the rate of decay more precisely, the literature has settled on a power-law decay. Accordingly, a second-order stationary process is said to exhibit long-range dependence in the time domain if its autocovariances are regularly varying at infinity, that is, if they behave asymptotically like a slowly varying function multiplied by a power law, with exponent greater than $-1$. An analogous definition exists for the spectral density: a time series is long-range dependent in the spectral domain if its spectral density has a singularity at the origin that can be written as the product of a slowly varying function and a power of the frequency (see \cite{binghamGoldie} for a detailed treatment of regular variation). Such power-law decay is closely related to the many scaling laws observed in empirical autocovariance structures across a wide range of domains, such as textiles \cite{cox1951}, agriculture \cite{whittle1956}, economics \cite{granger1966}, finance \cite{granger1995}, and market microstructure \cite{lilloFarmer2004,toth2015}.

Other second-order definitions are based on the rate of decay of the coefficients in the linear representation of the process, or on the growth rate of the variance of its partial sums; see, for example, \cite{pipirasTaqqu2017} (Chapter 2), \cite{giraitisKoulSurgailis} (Chapter 3), \cite{beran2013} (Chapter 1), and \cite{samorodnitsky2016} (Chapter 6). Overall, processes that are long-range dependent in the sense of their second-order properties have been observed in various fields like astronomy, hydrology, sociology, economics, and finance \cite{beran1994,beran2013}.

There has been some discussion regarding the equivalence of the various second-order definitions of long memory in the time and spectral domains. In the literature of the 1990s and early 2000s, many authors treated the time and spectral domain definitions as interchangeable, in the sense that the equivalence was considered ``operationally true,'' that is, valid in cases of practical interest, even though it was often stated in the form of a mathematical theorem \cite{gubner2005}.

These simplifications have since been clarified. It is now known that, in general, the time and spectral domain definitions of long memory are not equivalent, and several counterexamples have been constructed; see, for example, Examples 5.5 and 5.6 in \cite{samorodnitsky2007} and Exercise 2.6 in \cite{pipirasTaqqu2017}.

In \cite{gubner2005}, Gubner shows that the equivalence fails even when the slowly varying functions are asymptotically equal to positive constants. He provides a counterexample of a process that is long-range dependent in the time domain with an asymptotically constant slowly varying function, i.e., whose autocovariances asymptotically follow an exact power law with exponent greater than $-1$, but whose spectral measure is singular. As a result, the process admits no spectral density, and the spectral domain definition of long memory cannot be applied. Gubner also finds a counterexample to the converse: long-range dependence in the spectral domain with a slowly varying function asymptotically equal to a constant does not imply long-range dependence in the time domain with a slowly varying function asymptotically equal to a constant.

To obtain an equivalence between the time domain and the spectral domain, additional constraints must be imposed on the slowly varying functions that appear when expressing the autocovariances (respectively, the spectral density) as the product of a power of the lag (respectively, the frequency) and a slowly varying function. For example, it is enough to require them to belong to the Zygmund class \cite{gubner2005,samorodnitsky2007,beran2013}. In \cite{pipirasTaqqu2017}, Pipiras and Taqqu further relaxed this condition by showing that it is sufficient for the slowly-varying functions to be only quasi-monotone.

In this article,  we consider only second-order stationary time series whose spectral measures are absolutely continuous with respect to the Lebesgue measure. As a consequence, all processes under study admit a spectral density, and the definitions of long memory in the spectral domain will be formulated in terms of properties of this density rather than of the spectral measure.

Our goal is to contribute to a further rigorous treatment of second-order definitions of long memory. We look at the most general second-order definition of long memory, which requires only that the series of the absolute values of the autocovariances diverges. In the literature, one sometimes finds authors defining the spectral density as the sum of the trigonometric series whose coefficients are the autocovariances divided by $2\pi$, even for long-range dependent processes for which nothing is known about the behavior of the autocovariances at infinity, except that they decay slowly enough for their series to diverge. For instance, Section 4.1.3 of \cite{beran2013} presents some known results about spectral theory of---potentially long-range dependent---stationary sequences, while still defining the spectral density in the same section through this trigonometric series representation.

In this article, we show that the Fourier series of the spectral density of a generic long-range dependent process must be manipulated with great care, by providing an example of a long-memory process for which the series of the autocovariances in absolute value diverges, and whose spectral density exists but has an almost everywhere divergent Fourier series. To the best of our knowledge, no such example has been provided in the literature.

This is why, by spectral density, we mean the density of the spectral measure, rather than the trigonometric series whose coefficients are the autocovariances divided by $2\pi$. This distinction is necessary only for long-memory processes. If a process is short-range dependent, its autocovariances are absolutely summable, so that the Fourier series of its spectral density, denoted as $f$, is normally convergent, and the symmetric partial Fourier sums converge uniformly to a continuous function which is equal to $f$ $m$-almost everywhere, where $m$ denotes the Lebesgue measure. Consequently, there is a representative of the equivalence class of $f$ which is continuous on $\mathbb{R}$. Thus, for short-range dependent processes, studying the spectral measure is the same as studying the trigonometric series with autocovariances as coefficients.

For a long-memory process, we can only talk about conditional convergence of the Fourier series since the autocovariances of the process are not absolutely summable. In any case, the counterexample we provide in this article shows that for a long-memory process, the series may diverge at some points and may even diverge almost everywhere.

Notice that some highly efficient estimation techniques, like maximum likelihood estimation using Whittle's approximation, depend on the spectral density being sufficiently regular. For instance, spectral density estimation using the periodogram has been extensively studied for short-range dependent processes, cf \cite{brockwellDavisMethods} (Chapter 10), or \cite{beran2013} (Section 4.6). When the time series is linear and long-range dependent, but its spectral density takes the form $f(\lambda)=L(\lambda){|\lambda|}^{-2d}$ for all $\lambda \in [-\pi,\pi] \setminus \left\{0\right\}$, with $d \in (0,1/2)$ and $L$ continuous and bounded away from 0, some results can still be derived for the standardized discrete Fourier transform (DFT) and periodogram, like for instance their asymptotic distribution (\cite{giraitisKoulSurgailis}, Theorem 5.3.1). Likewise, the proof that the Whittle estimator is asymptotically normal requires similar conditions on the spectral density (\cite{giraitisKoulSurgailis}, Theorem 8.3.1). These restrictions on the spectral density imply that $f$ belongs to $L^{p}$ for some $p > 1$, which in turn guarantees that the Fourier series of $f$ converges almost everywhere to $f$ by the Carleson--Hunt theorem \cite{carleson1966,hunt1967}.

Even though the counterexample we present in this article is based on a constructive proof and can be simulated, it is somewhat exotic. In practical applications, processes most likely satisfy the regularity conditions assumed by the majority of estimation results and theorems.

The remainder of the paper is organized as follows. In Section \ref{section:hardyRogosinski}, we present our main result, which states that there exists a long-memory process whose spectral density has a Fourier series that diverges almost everywhere. To prove this result, we will explicitly construct such a process. Section \ref{section:construction} highlights the details of this construction. Section \ref{section:simulation} is devoted to empirical exploration, where we simulate this process and plot its autocovariances. In Section \ref{section:discussion}, we discuss stronger properties of long memory and the extent to which they guarantee that the spectral density can be legitimately written as a Fourier series. Finally, Appendix \ref{section:technicalProofs} contains the technical proofs.

\section{The Hardy--Rogosinski Process}
\label{section:hardyRogosinski}

We begin by introducing some definitions from ergodic theory, focusing only on the properties relevant to our process. The definitions provided below correspond to equivalent characterizations of general concepts related to ergodic theory, reformulated within the context of time series analysis. For an in-depth treatment of ergodic theory, as well as some equivalent characterizations of ergodicity, mixing, and weak mixing, the interested reader may refer to \cite{karlinTaylor} and \cite{silva2008}.

\Needspace{10\baselineskip}
\begin{definition}
\label{def:ergodicityMixing}
A strictly stationary real-valued stochastic process $(X_{t})_{t \in \mathbb{Z}}$ is said to be
\begin{itemize}[label=\upshape\textbullet, leftmargin=1.5em]
\item ergodic if for all $k \ge 1$ and $A,B \in \mathcal{B}\hspace{-1.6pt}\left(\rule{0pt}{9.1pt}\right.\hspace{-3.2pt}\mathbb{R}^{k}\hspace{-3.2pt}\left.\rule{0pt}{9.1pt}\right)$, the following quantity goes to 0 as $n \to +\infty$: $$\frac{1}{n}\sum_{t=1}^{n}\mathbb{P}((X_{t},\dotsc,X_{t+k-1}) \in A,(X_{1},\dotsc,X_{k}) \in B)-\mathbb{P}((X_{1},\dotsc,X_{k}) \in A)\mathbb{P}((X_{1},\dotsc,X_{k}) \in B),$$
\item weakly mixing if for all $k \ge 1$ and $A,B \in \mathcal{B}\hspace{-1.6pt}\left(\rule{0pt}{9.1pt}\right.\hspace{-3.2pt}\mathbb{R}^{k}\hspace{-3.2pt}\left.\rule{0pt}{9.1pt}\right)$, the following quantity goes to~0~as~$n \to +\infty$: $$\frac{1}{n}\sum_{t=1}^{n}\left|\mathbb{P}((X_{t},\dotsc,X_{t+k-1}) \in A,(X_{1},\dotsc,X_{k}) \in B)-\mathbb{P}((X_{1},\dotsc,X_{k}) \in A)\mathbb{P}((X_{1},\dotsc,X_{k}) \in B)\right|\hspace{-1.7pt},$$
\item mixing if for all $k \ge 1$ and $A,B \in \mathcal{B}\hspace{-1.6pt}\left(\rule{0pt}{9.1pt}\right.\hspace{-3.2pt}\mathbb{R}^{k}\hspace{-3.2pt}\left.\rule{0pt}{9.1pt}\right)$, we have $$\mathbb{P}((X_{t},\dotsc,X_{t+k-1}) \in A,(X_{1},\dotsc,X_{k}) \in B) \xrightarrow[t \to +\infty]{} \mathbb{P}((X_{1},\dotsc,X_{k}) \in A)\mathbb{P}((X_{1},\dotsc,X_{k}) \in B).$$
\end{itemize}
\end{definition}

\begin{remark}
The definition of ergodicity given in Definition \ref{def:ergodicityMixing} is equivalent to requiring that for all $k \ge 1$ and $A \in \mathcal{B}\hspace{-1.6pt}\left(\rule{0pt}{9.1pt}\right.\hspace{-3.2pt}\mathbb{R}^{k}\hspace{-3.2pt}\left.\rule{0pt}{9.1pt}\right)$, we have $$\frac{1}{n}\sum_{t=1}^{n}\mathds{1}_{\left\{(X_{t},\dotsc,X_{t+k-1}) \in A\right\}}  \xrightarrow[n \to +\infty]{\mathbb{P}\text{-a.s.}} \mathbb{P}((X_{1},\dotsc,X_{k}) \in A),$$ see for instance \cite{silva2008}, Chapter 5. In addition, it is not difficult to see that mixing implies weak mixing, which in turn implies ergodicity (see \cite{silva2008}, Lemma 6.2.1).

Moreover, a stationary Gaussian sequence is weakly mixing if and only if it is ergodic. We also have simpler characterizations in the Gaussian case: a stationary Gaussian sequence is weakly mixing if and only if its spectral measure is continuous, i.e., has no atoms, see \cite{cornfeldFomin} (Chapter 14, \S{2}, Theorem 1). Likewise, a stationary Gaussian sequence is mixing if and only if its autocovariances converge to 0 as the lag goes to infinity, see \cite{cornfeldFomin} (Chapter 14, \S{2}, Theorem 2). As a result, a sufficient condition for a stationary Gaussian sequence to be mixing is that its spectral measure is absolutely continuous.
\end{remark}

\begin{remark}
Additionally, the definition of mixing in Definition \ref{def:ergodicityMixing} should not be confused with the concept of \textit{strong} mixing. A strictly stationary process $(X_{t})_{t \in \mathbb{Z}}$ is \textit{strongly} mixing (or $\alpha$-mixing) if $$\alpha(n)=\sup_{A \in \mathcal{F}_{-\infty}^{0},\,B \in \mathcal{F}_{n}^{+\infty}}\left|\mathbb{P}(A \cap B)-\mathbb{P}(A)\mathbb{P}(B)\right|$$ converges to 0 as $n \to +\infty$, where $\mathcal{F}_{-\infty}^{0}=\sigma(X_{t},t \le 0)$ and $\mathcal{F}_{n}^{+\infty}=\sigma(X_{t},t \ge n)$. This notion is stronger than mixing as given in Definition \ref{def:ergodicityMixing}. Indeed, any strongly mixing stationary process is also mixing (see \cite{samorodnitsky2016}, Proposition 2.3.4). Moreover, a nondegenerate strictly stationary Gaussian sequence is strongly mixing if and only if it has an absolutely continuous spectral measure and its spectral density is $\log$-summable with a very specific form (see Theorem 7.1 in \cite{bradley2005}, and \cite{helsonSarason}).
\end{remark}

We now state the main theorem of this article.

\begin{theorem}
\label{thm:existenceProcessWithDivergentFourierSeries}
There exists a long-memory strictly stationary centered real-valued Gaussian process $(X_{t})_{t \in \mathbb{Z}}$ whose spectral measure is absolutely continuous with respect to the Lebesgue measure and whose spectral density is $\log$-summable and has an almost everywhere unboundedly divergent Fourier series.

As a result, this process is also purely nondeterministic and mixing (so it is weakly mixing and ergodic as well), and, denoting $f$ its spectral density, we have the following linear representations: $$X_{t}=\sum_{j=0}^{\infty}a_{j}\varepsilon_{t-j}~\text{for all}~t \in \mathbb{Z},\,\text{and}~f(\lambda)=\frac{\sigma_{0}^{2}}{2\pi}\,{\left|\sum_{j=0}^{\infty}a_{j}\e^{ij\lambda}\right|}^{2}~\text{for all}~\lambda \in \mathbb{R},$$ where $(a_{j})_{j \in \mathbb{N}}$ is a square-summable sequence with $a_{0}=1$, and $(\varepsilon_{t})_{t \in \mathbb{Z}}$ are i.i.d. Gaussian variables with mean zero and variance $\sigma_{0}^{2} > 0$.
\end{theorem}

\begin{remark}
Recall that for the Fourier series of an even integrable function $f$, \textit{unbounded divergence} almost everywhere means that for almost every $\theta$, the $n$-th symmetric partial Fourier sum of $f$ evaluated at $\theta$, which is real because $f$ is even, oscillates unboundedly as $n$ goes to infinity. More precisely, for almost every $\theta$, the limit superior of the symmetric partial Fourier sums of $f$, evaluated at $\theta$, is equal to $+\infty$, while the limit inferior is equal to $-\infty$ (see Theorem 5.7 in Chapter XIII, \S{5} of \cite{zygmund2003}).
\end{remark}

To prove Theorem \ref{thm:existenceProcessWithDivergentFourierSeries}, we will explicitely construct in Section \ref{section:construction} a process which satisfies the stated conditions. The general idea of the construction is to build an integrable function whose Fourier series diverges almost everywhere and that satisfies the properties of the spectral density of a real-valued stationary process, namely, that it is non-negative, even, and integrable on $(-\pi,\pi)$ (see the last remark of Section 4.3 in \cite{brockwellDavisMethods}).

Note that the first construction of an integrable function whose Fourier series diverges almost everywhere was made by Kolmogorov in 1922 \cite{kolmogorov1923}, when he was 19 years old! In 1926, he improved this result by constructing an integrable function whose Fourier series diverges at every point \cite{kolmogorov1926}.

Another example was constructed by Hardy and Rogosinski later on. Their proof uses similar ideas as Kolmogorov's, although the function itself is very different. The construction presented below is inspired by their work but has been adapted to obtain a valid spectral density for a real-valued process. This is why we coin this process the Hardy--Rogosinski process.

The construction of Hardy and Rogosinski can be found in \cite{hardyRogosinski} (Chapter VI, Section 6.2) and \cite{zygmund2003} (Chapter VIII, \S{3}), while \cite{ulyanov1983} gives a more detailed version of Kolmogorov's construction.

\section{Construction of the Process}
\label{section:construction}

Before getting into the core of the construction of the Hardy--Rogosinski process, which proves Theorem \ref{thm:existenceProcessWithDivergentFourierSeries}, we introduce some key definitions and notations related to Fourier analysis. These tools will enable us to construct a sequence of trigonometric polynomials for which we can find a lower bound on the modulus of the symmetric partial Fourier sums (see Lemma \ref{lemma:HardyTrigonometricPolynomials}).

We will use this sequence to construct an even real-valued non-negative integrable function $f$ whose Fourier series diverges unboundedly almost everywhere (see Theorem \ref{thm:existenceEvenFunctionFourierDiverges}). The properties of $f$ make it a good candidate to be the spectral density of a second-order stationary process. In fact, classical results in time series analysis ensure that there exists a Gaussian process whose spectral density is $f$. It will remain to show that this process satisfies the properties stated in Theorem \ref{thm:existenceProcessWithDivergentFourierSeries}.

\begin{notation}
\label{notation:symmetricPartialFourierSum}
Put $f$ a $2\pi$-periodic function from $\mathbb{R}$ to $\mathbb{C}$ such that $f \in L^{1}(\left(-\pi,\pi\right])$. For all $n \in \mathbb{N}$, for all $\theta \in \mathbb{R}$, we denote $S_{n}(f,\theta)$ the $n$-th symmetric partial Fourier sum associated with the function $f$, evaluated at the point $\theta$, i.e., $$S_{n}(f,\theta)=\sum_{l=-n}^{n}\hat{f}(l)\e^{\mathrm{i}l\theta},\,\text{where}~\hat{f}(l)=\frac{1}{2\pi}\int_{-\pi}^{\pi}f(\theta)\e^{-\mathrm{i}l\theta}d\theta.$$
\end{notation}

We also define the Dirichlet and Fejér kernels, which are crucial in Hardy and Rogosinski's construction. For a comprehensive analysis of these kernels and their link to Fourier series, the interested reader may refer to \cite{hardyRogosinski} (Chapter III, Section 3.6), and \cite{zygmund2003} (Chapter I, \S{1} and Chapter III, \S{3}).

\begin{definition}
For all $n \in \mathbb{N}$, the Dirichlet kernel $D_{n}$ and the Fejér kernel $F_{n}$ are the trigonometric polynomials of degree $n$ defined for all $\theta \in \mathbb{R}$ by
\begin{flalign*}
D_{n}(\theta) &= \frac{1}{2}\sum_{l=-n}^{n}\e^{\mathrm{i}l\theta}=\frac{1}{2}+\sum_{l=1}^{n}\cos(l\theta)=\frac{\sin\hspace{-1.6pt}\left(\left(n+\frac{1}{2}\right)\theta\right)}{2\sin\hspace{-1.6pt}\left(\frac{\theta}{2}\right)}, \\
F_{n}(\theta) &= \frac{1}{n+1}\sum_{l=0}^{n}D_{l}(\theta)=\frac{2}{n+1}{\left(\frac{\sin\hspace{-1.6pt}\left(\frac{1}{2}(n+1)\theta\right)}{2\sin\hspace{-1.6pt}\left(\frac{\theta}{2}\right)}\right)}^{2}=\frac{1}{2}\sum_{l=-n}^{n}\left(1-\frac{|l|}{n+1}\right)\e^{\mathrm{i}l\theta}.
\end{flalign*}
\end{definition}

\begin{lemma}
\label{lemma:propertiesFejerKernel}
For all integer $n$, $F_{n}$ satisfies the following properties: $$F_{n} \ge 0,\,F_{n}~\text{is even},\,\text{and}~\frac{1}{\pi}\int_{-\pi}^{\pi}F_{n}(\theta)d\theta=1.$$ Moreover, there is a constant $C > 0$ such that for all $n \in \mathbb{N}$ and for all $\theta \in (0,\pi]$, we have $$F_{n}(\theta) \le \frac{C}{(n+1)\theta^{2}}.$$
\end{lemma}

\begin{proof}
This is a known result; see, for instance, \cite{zygmund2003} (Chapter III, \S{3}).
\end{proof}

The kernel of the construction lies in the preliminary lemma which follows. Its proof is rather technical and is therefore deferred to Appendix \ref{subsection:appendixHardy}. The proof is constructive, in the sense that all quantities and sequences appearing in the lemma are explicitly defined. These explicit constructions are reused in Section \ref{section:simulation}---with minor numerical adjustments---to compute the theoretical autocovariances of the process and simulate it.

\Needspace{10\baselineskip}
\begin{lemma}
\label{lemma:HardyTrigonometricPolynomials}
There exists
\vspace{-\parskip}
\begin{itemize}
\item a sequence $(\varphi_{n})_{n \in \mathbb{N}}$ of non-negative even trigonometric polynomials with constant term $1/2$,
\item a real sequence $(M_{n})_{n \in \mathbb{N}}$ of positive numbers diverging to infinity,
\item and a sequence $(E_{n})_{n \in \mathbb{N}}$ of Borel subsets of $[-\pi,\pi]$ whose Lebesgue measure converges to $2\pi$,
\end{itemize}
\vspace{-\parskip}
such that for all $n \in \mathbb{N}$, for all $\theta \in E_{n}$, there exists $p_{n}(\theta) \in \mathbb{N}$ such that $\max(1,n) \le p_{n}(\theta) \le q_{n}$,\,and $$\left|S_{p_{n}(\theta)}(\varphi_{n},\theta)\right| \ge M_{n},$$ where $q_{n}$ is the degree of the polynomial $\varphi_{n}$.
\end{lemma}

The sequences provided by Lemma \ref{lemma:HardyTrigonometricPolynomials} may serve as the building blocks for constructing a candidate spectral density whose Fourier series diverges almost everywhere. The next theorem formalizes this result. Its proof is constructive and follows that of Theorem 3.1 in Chapter VIII, \S{3} of \cite{zygmund2003}, and is presented in details in Appendix \ref{subsection:appendixEvenFourier}.

In brief, the idea of the construction is to extract a subsequence of the polynomials $(\varphi_{n})_{n \in \mathbb{N}}$ and rescale and shift each element so as to obtain non-overlapping polynomials. The resulting polynomials are normalized by the square roots of the $M_{n}$ to ensure they are summable almost everywhere. Let $(\psi_{k})_{k \ge 0}$ denote these polynomials, and let $f$ be their sum. Then $f$ is an even, real-valued, non-negative function that is integrable on $\left[-\pi,\pi\right]$. Since the polynomials $(\psi_{k})_{k \ge 0}$ do not overlap, the non-zero Fourier coefficients of $f$ are precisely the non-zero Fourier coefficients of the polynomials, written successively. Finally, using the sets $(E_{n})_{n \in \mathbb{N}}$, it is possible to construct a Borel subset of $[-\pi,\pi]$ with measure $2\pi$ on which the Fourier series of $f$ diverges unboundedly.

This leads to an apparent paradox, whereby $f$ can be expressed as a convergent sum of trigonometric polynomials, with its non-zero Fourier coefficients inherited from those polynomials, yet these same Fourier coefficients generate a trigonometric series that diverges almost everywhere. Since the trigonometric polynomials are summable, one might have expected the Fourier series to converge as well, as both share the same non-zero coefficients.

The explanation is that the sequence $\left(\sum_{j=0}^{k}\psi_{j}(\theta)\right)$, which converges for almost every $\theta$ to $f(\theta)$, is a convergent subsequence of the symmetric partial sums of the trigonometric series whose coefficients are those of the polynomials $(\psi_{k})_{k \ge 0}$. Meanwhile, this trigonometric series diverges unboundedly almost everywhere. \textbf{To sum up, $f$ can be seen as the limit of a convergent subsequence of the symmetric partial sums of a trigonometric series that diverges unboundedly almost everywhere. The Fourier series of $f$ coincides exactly with this divergent trigonometric series.}

\begin{theorem}
\label{thm:existenceEvenFunctionFourierDiverges}
There exists an even real-valued non-negative function $f$ integrable on $\left[-\pi,\pi\right]$ whose Fourier series diverges unboundedly almost everywhere.
\end{theorem}

We can now finish the construction of the Hardy--Rogosinski process and prove Theorem \ref{thm:existenceProcessWithDivergentFourierSeries}.

\begin{proof}[Theorem \ref{thm:existenceProcessWithDivergentFourierSeries}]
Put $\varepsilon > 0$, and denote $f$ the sum of the function obtained in Theorem \ref{thm:existenceEvenFunctionFourierDiverges} and $\varepsilon$. Thus, $f$ is even, real-valued, integrable, and such that $f \ge \varepsilon$ on $[-\pi,\pi]$. Moreover, we have $\log(\varepsilon) \le \log(f) \le f$ on $[-\pi,\pi]$. Since $f$ is integrable, it follows that $f$ is $\log$-summable. Additionally, the Fourier series of $f$ is the same as the Fourier series of the function given by Theorem \ref{thm:existenceEvenFunctionFourierDiverges}, except that its constant term is greater by epsilon. Therefore, the Fourier series of $f$ diverges unboundedly almost everywhere.

For all $h \in \mathbb{Z}$, put $\gamma(h)$ the $h$-th Fourier coefficient of $f$ multiplied by $2\pi$. Since $f$ is even, we obtain
\begin{equation}
\label{eq:covariancesConstructedFromf}
\gamma(h)=\int_{-\pi}^{\pi}f(\theta)\e^{-\mathrm{i}h\theta}d\theta=\int_{-\pi}^{\pi}f(\theta)\e^{\mathrm{i}h\theta}d\theta.
\end{equation}
Since $f$ is even and real-valued, it is clear that $\gamma$ is an even real-valued function defined on $\mathbb{Z}$.

Moreover, put $F$ the function defined from $[-\pi,\pi]$ to $\mathbb{R}$ such that for all $\lambda \in [-\pi,\pi]$, $$F(\lambda)=\int_{-\pi}^{\lambda}f(\theta)d\theta.$$ Then, $F$ is continuous, non-decreasing, and bounded on $[-\pi,\pi]$, with $F(-\pi)=0$. Moreover, we have $$\gamma(h)=\int_{\left(-\pi,\pi\right]}\e^{\mathrm{i}h\theta}dF(\theta),$$ for all $h \in \mathbb{Z}$, thanks to equality (\ref{eq:covariancesConstructedFromf}). By Herglotz's theorem (see Theorem 4.3.1 in \cite{brockwellDavisMethods}), it follows that the function $\gamma$ is non-negative definite. The proof of Theorem 1.5.1 in \cite{brockwellDavisMethods} then guarantees that there exists a strictly stationary centered real-valued Gaussian sequence $(X_{t})_{t \in \mathbb{Z}}$ whose autocovariance function is $\gamma$, and equality (\ref{eq:covariancesConstructedFromf}) ensures that its spectral measure is absolutely continuous with respect to the Lebesgue measure, with spectral density $f$. We refer to the process $(X_{t})_{t \in \mathbb{Z}}$ as the Hardy--Rogosinski process. Since the Fourier series of $f$ diverges at some points, the autocovariances $\gamma(h)$, $h \in \mathbb{Z}$, are not absolutely summable, and $(X_{t})_{t \in \mathbb{Z}}$ is a long-memory process.

Additionally, since $f$ is log-summable, Kolmogorov's formula (see \cite{brockwellDavisMethods}, Section 5.8, and \cite{hannan}, Chapter III, Section 3, Theorem 3) guarantees that the sequence $(X_{t})_{t \in \mathbb{Z}}$ is not deterministic (see \cite{brockwellDavisMethods}, Section 5.7, for the definition of a deterministic process).

Hence, the process $(X_{t})_{t \in \mathbb{Z}}$ is a strictly stationary centered nondeterministic Gaussian process. According to \cite{brockwellDavisMethods} (Theorem 5.7.1), \cite{davisGaussian}, and \cite{ihara} (Section 2.2), Wold's decomposition allows $X_{t}$ to be expressed as $$X_{t}=\sum_{j=0}^{\infty}a_{j}\varepsilon_{t-j}+V_{t},\,\forall t \in \mathbb{Z},$$ where $(\varepsilon_{t})_{t \in \mathbb{Z}}$ are i.i.d. Gaussian variables with mean zero and variance $\sigma_{0}^{2} > 0$, $(a_{j})_{j \in \mathbb{N}}$ is a square-summable sequence with $a_{0}=1$, and $(V_{t})_{t \in \mathbb{Z}}$ is a deterministic Gaussian process independent of~$(\varepsilon_{t})_{t \in \mathbb{Z}}$.

Moreover, the process $(X_{t})_{t \in \mathbb{Z}}$ has an absolutely continuous spectral measure and a log-summable spectral density. It follows that $(X_{t})_{t \in \mathbb{Z}}$ is purely nondeterministic (i.e., it is not deterministic and the deterministic part in its Wold's decomposition is zero), see for instance, \cite{skorokhodStochastic} (Theorem 2 in Chapter IV, Section 7), \cite{shiryayevKolmogorov} (Section 27.7, Theorem 22), \cite{giraitisKoulSurgailis} (Theorem 3.2.1), or \cite{ihara} (Section 2.2). Thus, we have $V_{t}=0$ for all $t \in \mathbb{Z}$. In addition, \cite{hannan} (Chapter III, Section 3, Theorem 4) and \cite{skorokhodStochastic} (Lemma 1 in Chapter IV, Section 7) give us the desired form for the spectral density of $(X_{t})_{t \in \mathbb{Z}}$.

Finally, the sequence $(X_{t})_{t \in \mathbb{Z}}$ is a stationary Gaussian sequence with an absolutely continuous spectral measure, so it is mixing, and hence weakly mixing and ergodic as well.
\end{proof}

\begin{remark}
A second-order stationary sequence $X$ is purely nondeterministic if and only if its spectral distribution function is absolutely continuous and its spectral density $f$ is $\log$-summable, see \cite{skorokhodStochastic} (Theorem 2 in Chapter IV, Section 7), \cite{shiryayevKolmogorov} (Section 27.7, Theorem 22), or \cite{ihara} (Section 2.2). In addition, if a second-order stationary process $X$ has a spectral distribution function which is absolutely continuous, then it is non deterministic if and only if its spectral density is log-summable, see Kolmogorov's formula in \cite{brockwellDavisMethods} (Section 5.8) and \cite{hannan} (Chapter III, Section 3, Theorem 3).

Thus, if a second-order stationary sequence $X$ has an absolutely continuous spectral distribution function, it is either purely nondeterministic or deterministic. Note that there are examples of deterministic sequences with absolutely continuous spectral distribution functions, see Example 5.6.1 in \cite{brockwellDavisMethods}.
\end{remark}

\begin{remark}
It may be possible to strengthen the result of Theorem \ref{thm:existenceProcessWithDivergentFourierSeries} by saying that the Fourier series of the spectral density of $(X_{t})_{t \in \mathbb{Z}}$ diverges unboundedly everywhere. To construct such a spectral density, one could leverage the proofs of Theorem 4.1 and Lemma 4.2 in Chapter VIII, \S{4} of \cite{zygmund2003}, which show that there is an integrable function $f$ whose Fourier series diverges unboundedly everywhere. The construction of Zygmund relies on ideas similar to those in the proofs of Lemma \ref{lemma:HardyTrigonometricPolynomials} and Theorem \ref{thm:existenceEvenFunctionFourierDiverges}, although it is more complex.
\end{remark}

\section{Simulation and Empirical Autocovariances}
\label{section:simulation}

As a numerical experimentation, we want to plot the first few autocovariances of the process $(X_{t})_{t \in \mathbb{Z}}$ and simulate some trajectories of this process. To achieve this, we use the same constructions as in the proofs of Lemma \ref{lemma:HardyTrigonometricPolynomials} and Theorem \ref{thm:existenceEvenFunctionFourierDiverges}. However, since many quantities involved in the construction of the spectral density of $(X_{t})_{t \in \mathbb{Z}}$ grow exponentially fast, we make some minor adjustments.

We start by introducing some notations used in the proof of Lemma \ref{lemma:HardyTrigonometricPolynomials}. For all $n \ge 50$, put $$x_{n}=\left\lfloor\frac{4n+1}{4}\hspace{-1.6pt}\left(1-\frac{1}{{(\log(n))}^{\frac{1}{6}}}\right)\right\rfloor~\text{and}~M_{n}=\frac{1}{16\pi^{2}}\frac{4n+1}{x_{n}}{(\log(n))}^{\frac{1}{2}}-\frac{\pi^{2}}{2}.$$ Compared with the proof of Lemma \ref{lemma:HardyTrigonometricPolynomials}, we replaced the constant $C$ in the definition of $M_{n}$ with $\pi^{2} / 2$, as it can be easily shown that it is a valid value for $C$ so that the upper bound of the Fejér kernel in Lemma \ref{lemma:propertiesFejerKernel} holds.

For all $l=1,\dotsc,x_{n}$, define $A_{l}^{n}$ as $$A_{l}^{n}=\frac{4\pi{l}}{4n+1}.$$ Let $m_{1}^{n},\dotsc,m_{x_{n}}^{n}$ denote some integers defined by putting $m^{(0)}=m_{l}^{n}$ if $l > 0$ (and $m_{l}^{n}$ has been constructed), $m^{(0)}=n^{4}$ otherwise, and $$m_{l+1}^{n}=\frac{(4n+1)k-1}{2},$$ where $k$ is the smallest odd integer larger than the quotient of the Euclidean division of $2m^{(0)}+1$ by $4n+1$.

Looking at the proof of Lemma \ref{lemma:HardyTrigonometricPolynomials}, the sequence of trigonometric polynomials $(\varphi_{n})$ is such that
\begin{flalign}
\varphi_{n}(\theta) &= \frac{1}{2x_{n}}\sum_{l=1}^{x_{n}}\sum_{r=-m_{l}^{n}}^{m_{l}^{n}}\left(1-\frac{|r|}{m_{l}^{n}+1}\right)\cos(rA_{l}^{n})\e^{\mathrm{i}r\theta} \notag \\
&= \frac{1}{2x_{n}}\sum_{r=-m_{x_{n}}^{n}}^{m_{x_{n}}^{n}}\left(\hspace{1.7pt}\sum_{l=1}^{x_{n}}\left(1-\frac{|r|}{m_{l}^{n}+1}\right)\cos(rA_{l}^{n})\mathds{1}_{\left\{|r| \le m_{l}^{n}\right\}}\right)\e^{\mathrm{i}r\theta}. \label{eq:rewritePolynomialPhin}
\end{flalign}

We now make the following adjustment compared to the proof of Lemma \ref{lemma:HardyTrigonometricPolynomials}: we consider all the values of $M_{n}$ instead of starting from $n_{0}$ such that $M_{n} > 0$ for all $n \ge n_{0}$. More precisely, we define by induction an increasing sequence $(n_{k})_{k \ge 1}$ of integers such that for all $k \ge 1$, we have $$\frac{1}{M_{n_{k}}+\frac{\pi^{2}}{2}} < \frac{10}{2^{k}}.$$ Instead of the formula (\ref{eq:definitionFunctionfFourierDivergent}) in the proof of Theorem \ref{thm:existenceEvenFunctionFourierDiverges}, we define $f(\theta)$ as
\begin{equation}
\label{eq:newDefinitionfTheta}
f(\theta)=\frac{1}{2\pi}\left(5-\frac{1}{2}\sum_{k=1}^{\infty}\frac{1}{\sqrt{M_{n_{k}}+\frac{\pi^{2}}{2}}}+\sum_{k=1}^{\infty}\frac{\varphi_{n_{k}}(c_{k}\theta)}{\sqrt{M_{n_{k}}+\frac{\pi^{2}}{2}}}\right)\hspace{-1.7pt},
\end{equation}
where $(c_{k})_{k \ge 1}$ is a sequence of odd integers such that for all $k \ge 1$, $q_{n_{k}}c_{k} < c_{k+1}$, with $q_{n_{k}}$ the degree of the polynomial $\varphi_{n_{k}}$.

With this new definition, $f$ is well-defined even if some values of the sequence $(M_{n})_{n \in \mathbb{N}}$ are negative. Thus, unlike in the proof of Lemma \ref{lemma:HardyTrigonometricPolynomials}), we don't need to find an integer $n_{0} \ge 50$ such that for all $n \ge n_{0}$, $M_{n} > 0$. This integer $n_{0}$ is indeed enormous, making numerical experiments impractical. With our adjustment, we can just put $n_{0}=50$. Moreover, the reasoning of the proof of Theorem \ref{thm:existenceEvenFunctionFourierDiverges} still holds, and the first few polynomials $\varphi_{0},\varphi_{1},\varphi_{2},\dotsc$ don't involve overly large numbers. Note also that we have $\hat{f}(0)=5/(2\pi)$.

Under this adjusted construction, we can take $n_{1}=500$, $n_{2}=10000$, $c_{1}=1$, and put $c_{2}$ the smallest odd integer larger than $m_{x_{n_{1}}}^{n_{1}}$. In fact, calculations show that $m_{x_{n_{1}}}^{n_{1}}$ is odd, and we then take $c_{2}=m_{x_{n_{1}}}^{n_{1}}+2$.

Now, from the definition of $f(\theta)$ in (\ref{eq:newDefinitionfTheta}), we see that
\begin{equation}
\label{eq:fAsASumOfPsikMain}
f(\theta)=\sum_{k=0}^{\infty}\psi_{k}(\theta),
\end{equation}
where $$\psi_{k}(\theta)=
\begin{cases}
\vspace{5pt}
\displaystyle\frac{5}{2\pi} & \text{if}~k=0, \\
\displaystyle\frac{1}{2\pi}\frac{\varphi_{n_{k}}(c_{k}\theta)-\frac{1}{2}}{\sqrt{M_{n_{k}}+\frac{\pi^{2}}{2}}} & \text{if}~k \ge 1.
\end{cases}$$ Moreover, it follows from equality (\ref{eq:rewritePolynomialPhin}) that for all $k \ge 1$,
\begin{equation}
\label{eq:psikTrigonometricExpansion}
\psi_{k}(\theta)=\sum_{\substack{r=-m_{x_{n_{k}}}^{n_{k}} \\ r \ne 0}}^{m_{x_{n_{k}}}^{n_{k}}}\left(\frac{1}{2\pi}\frac{1}{\sqrt{M_{n_{k}}+\frac{\pi^{2}}{2}}}\frac{1}{2x_{n_{k}}}\sum_{l=1}^{x_{n_{k}}}\left(1-\frac{|r|}{m_{l}^{n_{k}}+1}\right)\cos(rA_{l}^{n_{k}})\mathds{1}_{\left\{|r| \le m_{l}^{n_{k}}\right\}}\right)\e^{\mathrm{i}rc_{k}\theta}.
\end{equation}
Following the discussion regarding the Fourier coefficients of $f$ in the proof of Theorem \ref{thm:existenceEvenFunctionFourierDiverges}, we know that for all $k \ge 1$, the coefficients of $\psi_{k}$ with degrees $c_{k}$ to $m_{x_{n_{k}}}^{n_{k}}c_{k}$ are exactly the $c_{k}\text{-th},\dotsc,m_{x_{n_{k}}}^{n_{k}}c_{k}\text{-th}$ Fourier coefficients of $f$, respectively. Because $\gamma(h)=2\pi\hat{f}(h)$, it follows that for all $h=c_{k},\dotsc,m_{x_{n_{k}}}^{n_{k}}c_{k}$, $$\gamma(h)=\begin{cases}
\vspace{5pt}
\displaystyle\frac{1}{\sqrt{M_{n_{k}}+\frac{\pi^{2}}{2}}}\frac{1}{2x_{n_{k}}}\sum_{l=1}^{x_{n_{k}}}\left(1-\frac{h}{m_{l}^{n_{k}}+1}\right)\cos(hA_{l}^{n_{k}})\mathds{1}_{\left\{h \le m_{l}^{n_{k}}\right\}} & \text{if}~h~\text{is a multiple of}~c_{k}, \\
\displaystyle 0 & \text{otherwise}.
\end{cases}$$ Moreover, $\gamma(0)=5$, and if $h > 0$ and $h$ is not equal to one of $c_{k},\dotsc,m_{x_{n_{k}}}^{n_{k}}c_{k}$ for any $k \ge 1$, then $\gamma(h)=0$. Thus, we have an explicit formula for $\gamma(h)$, whatever the value of $h \in \mathbb{Z}$. This means that we can plot the autocovariance function of the process $(X_{t})_{t \in \mathbb{Z}}$, and we can even simulate this process.

\begin{figure}[htbp]
\captionsetup[subfigure]{skip=0pt}
    \centering
    \caption{Theoretical autocovariances of the process $(X_t)_{t \in \mathbb{Z}}$.\vspace{12pt}}

    \begin{subfigure}{\linewidth}
        \centering
        \caption{Autocovariances from lag $1$ to lag $m_{x_{n_{1}}}^{n_{1}}$}
        \label{figure:autocovariancesc1}
        \def\svgwidth{22cm}
        \resizebox{\linewidth}{!}{
            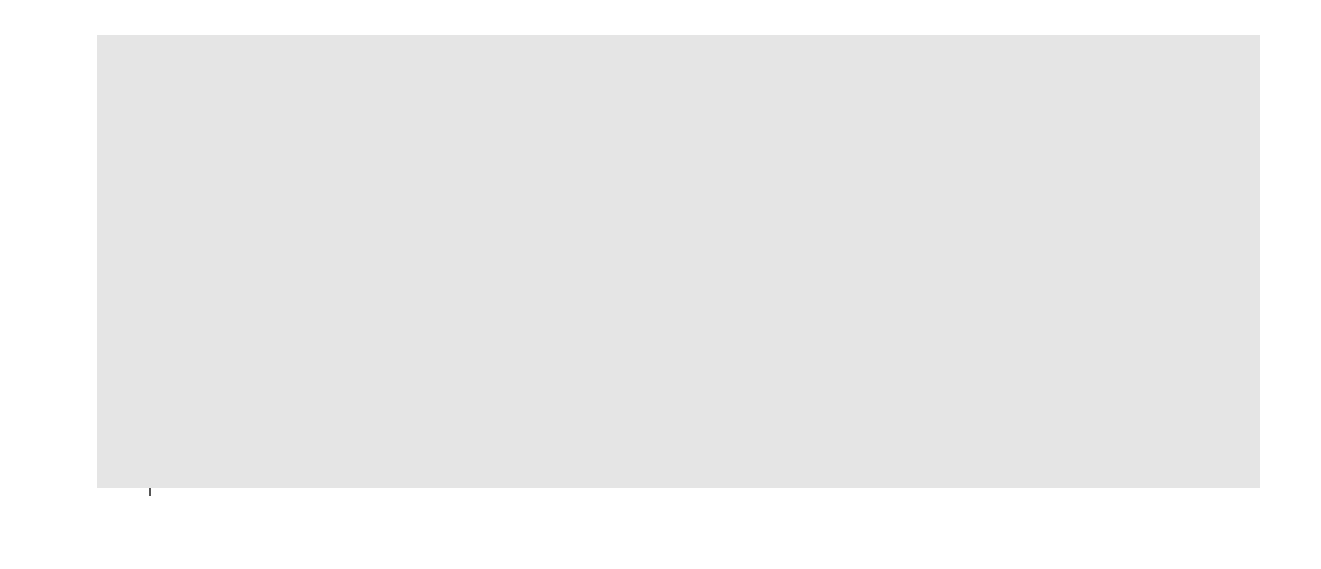
        }
    \end{subfigure}

    \vspace{2em}

    \begin{subfigure}{\linewidth}
        \centering
        \caption{Autocovariances from lag $c_{2}=m_{x_{n_{1}}}^{n_{1}}+2$ to lag $m_{x_{n_{2}}}^{n_{2}}c_{2}$}
        \label{figure:autocovariancesc2}
        \def\svgwidth{22cm}
        \resizebox{\linewidth}{!}{
            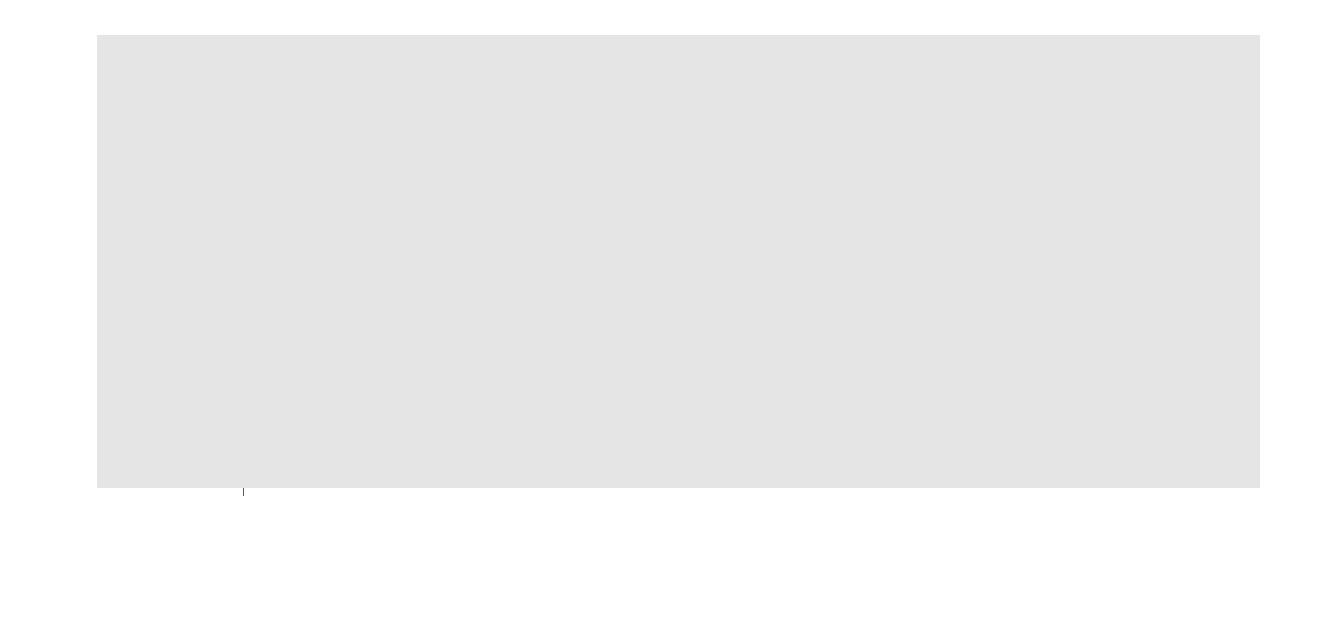
        }
    \end{subfigure}

\vspace{5pt}
\fontsize{10}{12}\selectfont
It may not be evident from the above graph, but only the autocovariances at lags that are multiples of $c_{2}$ and are between $c_{2}$ and $m_{x_{n_{2}}}^{n_{2}}c_{2}$ may be nonzero. All others are zero. Thus, among the $(m_{x_{n_{2}}}^{n_{2}}-1)c_{2}+1$ autocovariances with lags from $c_{2}$ to $m_{x_{n_{2}}}^{n_{2}}c_{2}$, at most $m_{x_{n_{2}}}^{n_{2}}$ are nonzero.
\end{figure}

Figure \ref{figure:autocovariancesc1} and Figure \ref{figure:autocovariancesc2} show the autocovariances of the process $(X_{t})_{t \in \mathbb{Z}}$ from lag $1$ to lag $m_{x_{n_{2}}}^{n_{2}}c_{2}$. More precisely, Figure \ref{figure:autocovariancesc1} displays the autocovariances from lag $1$ to lag $m_{x_{n_{1}}}^{n_{1}}$, while Figure \ref{figure:autocovariancesc2} displays the autocovariances from lag $c_{2}=m_{x_{n_{1}}}^{n_{1}}+2$ to lag $m_{x_{n_{2}}}^{n_{2}}c_{2}$ (the autocovariance at lag $m_{x_{n_{1}}}^{n_{1}}+1$ is equal to zero). In other words, these figures represent the first two connected blocks as described in the proof of Theorem \ref{thm:existenceEvenFunctionFourierDiverges}, i.e., they show the Fourier coefficients of $\psi_{1}$ and $\psi_{2}$, respectively, multiplied by $2\pi$.

We observe that the autocovariances of $(X_{t})_{t \in \mathbb{Z}}$ exhibit oscillations which take both negative and positive values. Additionally, the connected block associated with $\psi_{k}$, contains more and more terms as $k$ increases. In the meantime, the blocks are also filled with more and more zeros, due to the non-overlapping operation obtained through the multiplication by $c_{k}$. Yet, the oscillations are such that the $k$-th block presents spikes with a period approximately equal to $2n_{k}$. In particular, this period extends as $k$ increases. These spikes are most prominent at the beginning of the blocks and gradually diminish until they disappear entirely at the end of the blocks. Hence, the autocovariances at the end of the blocks are very close to zero or actually zero.

At the same time, the height of the spikes at the beginning of the blocks decreases as $k$ increases, thanks to the division by $\sqrt{M_{n_{k}}+\frac{\pi^{2}}{2}}$, a quantity that tends to $+\infty$ as $k$ goes to $+\infty$. This ensures that the autocovariances go to zero when the lag goes to infinity, which is a necessary condition because they are the Fourier coefficients of an integrable function (cf the Riemann--Lebesgue lemma).

Finally, notice that in these two figures, the quantities involved are gigantic. For instance, $m_{1}^{n_{1}}$ in the polynomial $\psi_{n_{1}}$ has order of magnitude $10^{10}$, while $m_{66}^{n_{1}}$, $m_{110}^{n_{1}}$, and $m_{x_{n_{1}}}^{n_{1}}$, which appear in Figure \ref{figure:autocovariancesc1}, have orders $10^{30}$, $10^{43}$, and $10^{49}$, respectively. The number $c_{2}$ is also on the order of $10^{49}$. On the other hand, $m_{1000}^{n_{2}}c_{2}$, $m_{2500}^{n_{2}}c_{2}$, and $m_{x_{n_{2}}}^{n_{2}}c_{2}$, which appear in Figure \ref{figure:autocovariancesc2}, have orders $10^{366}$, $10^{818}$, and $10^{996}$, respectively. Yet, the adjustments we've made make it possible to calculate the first few autocovariances efficiently. The calculations can be quickly performed using Python, which can handle very large integers, and by leveraging the $2\pi$-periodicity of the cosine function. However, the subsequent blocks, associated with the polynomials $\psi_{k}$ for $k \ge 3$, involve much larger numbers due to the very rapid growth of $(n_{k})_{k \ge 1}$, which is a result of the extremely slow increase of the sequence $(M_{n})_{n \in \mathbb{N}}$.

Now, let us simulate the process $(X_{t})_{t \in \mathbb{Z}}$. Since this is a Gaussian process and we can calculate its autocovariances, the task is quite straightforward, using for instance the Cholesky decomposition of the covariance matrix. More specifically, for $n \ge 1$, the covariance matrix $\Sigma_{n}$ of $(X_{t})_{1 \le t \le n}$ is the Toeplitz matrix where both the first column and first row are $(\gamma(0),\dotsc,\gamma(n-1))$, with $\gamma(h)$ denoting the autocovariance at lag $h$ of the process $(X_{t})_{t \in \mathbb{Z}}$. We then set $Y=LZ$, where $\Sigma_{n}=LL^\prime$ is the Cholesky decomposition of $\Sigma_{n}$, and $Z$ is a centered Gaussian vector of size $n$ with the identity matrix as its covariance matrix. It follows that the random vector $Y$ has the same distribution as $(X_{t})_{1 \le t \le n}$.

Figure \ref{figure:empiricalAutocovariances} shows the empirical autocovariances averaged over 1000 simulations of length 100000 of the Hardy--Rogosinski process, together with the corresponding theoretical autocovariances for comparison. For readability, the $y$-axis is truncated at 1.25, although the autocovariance at lag 0, $\gamma(0)$, is actually equal to 5. As expected, the empirical values are close to the theoretical ones, since the process $(X_{t})_{t \in \mathbb{Z}}$ is strictly stationary and ergodic. This implies that its empirical autocovariances converge almost surely towards their theoretical counterparts.

\vspace{5pt}
\begin{figure}[htbp]
    \centering
    \caption{Empirical autocovariances from 1000 simulations of a 100000-step trajectory}
    \label{figure:empiricalAutocovariances}
    \resizebox{0.7\linewidth}{!}{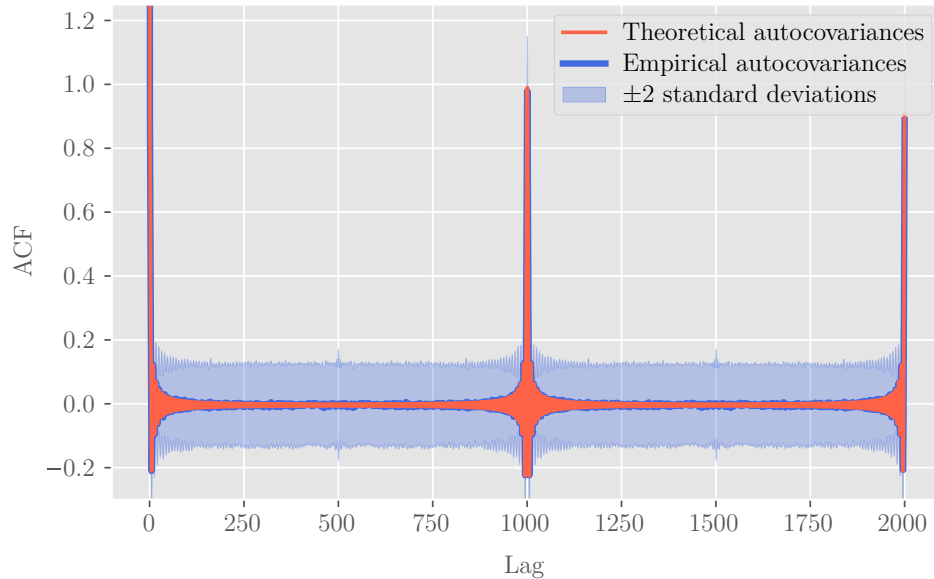}
\end{figure}

The construction of the spectral density of the process $(X_{t})_{t \in \mathbb{Z}}$ might seem a bit artificial. Indeed, the autocovariances have a rather unconventional form that doesn't correspond to any specific physical phenomenon. Hence, while it is not merely a theoretical construct since we can simulate this process, its usefulness appears to be mostly from a theoretical perspective.

Finally, notice that it is not possible to test whether the spectral density of a process has a convergent or divergent Fourier series, just by knowing a finite number of its autocovariances. For instance, consider the centered Gaussian process $(Y_{t})_{t \in \mathbb{Z}}$ whose spectral density is obtained using the same construction as the spectral density of $(X_{t})_{t \in \mathbb{Z}}$, but summing only $\psi_{0},\dotsc,\psi_{p}$ in equality (\ref{eq:fAsASumOfPsikMain}), instead of all the polynomials $(\psi_{k})_{k \ge 0}$. Then, the first few autocovariances of $(X_{t})_{t \in \mathbb{Z}}$ and $(Y_{t})_{t \in \mathbb{Z}}$ are the same. However, the spectral density of $(Y_{t})_{t \in \mathbb{Z}}$ is a trigonometric polynomial and admits a Fourier representation, while the Fourier series of the spectral density of $(X_{t})_{t \in \mathbb{Z}}$ diverges almost everywhere.

\section{Discussion and Stronger Definitions of Long Memory}
\label{section:discussion}

Since the condition that autocovariances are not absolutely summable is too broad, stronger definitions of long memory have been proposed in the literature. The most common definitions fall under the so-called second-order theory of long-range dependence. They focus on properties that are expressed either in terms of the autocovariances (time domain) or the spectral measure (spectral~domain).

In order to have conditions that are general enough, these definitions involve slowly varying functions, a concept that originated in the study of regular variation in the 1930s, particularly through the work of Jovan Karamata. We first introduce some key definitions related to regular variation. The interested reader may refer to \cite{binghamGoldie}, Chapter 1, for a comprehensive treatment of slowly varying~functions.

\begin{definition}
Put $a \in \mathbb{R}$ and let $L$ be a measurable real function defined on $[a,+\infty)$. The function $L$ is said to be slowly varying at infinity (in Karamata's sense) if $L$ is positive on some neighbourhood of infinity and for all $\lambda > 0$, we have $$\frac{L(\lambda{x})}{L(x)} \underset{x \to +\infty}{\longrightarrow} 1.$$
\end{definition}

\begin{remark}
This definition can be extended to consider functions that are slowly varying at some specific point in $\mathbb{R}$. For instance, a measurable function $L$ defined on $(0,a]$, where $a > 0$, is said to be slowly varying at zero if it is positive on some neighbourhood of 0 and the function $x \mapsto L\hspace{-1.6pt}\left(\frac{1}{x}\right)$, defined on $[1/a,+\infty)$, is slowly varying at infinity.
\end{remark}

\begin{definition}
\label{def:quasiMonotoneFunction}
Put $a \in \mathbb{R}$ and let $L$ be a measurable real function defined on $[a,+\infty)$ such that $L$ is locally of bounded variation and slowly varying. $L$ is said to be quasi-monotone if for some $\delta > 0$, $$\int_{a}^{x}u^{\delta}\left|dL(u)\right| \underset{x \to +\infty}{=} O\hspace{-1.6pt}\left(\rule{0pt}{9.1pt}\right.\hspace{-3.2pt}x^{\delta}L(x)\hspace{-3.2pt}\left.\rule{0pt}{9.1pt}\right)\hspace{-1.7pt}.$$
\end{definition}

\begin{remark}
According to Corollary 2.8.2 in \cite{binghamGoldie}, ``for some $\delta > 0$'' can be replaced with ``for some $\delta < 0$'' in Definition \ref{def:quasiMonotoneFunction}. Moreover, according to Theorem 2.7.2 and the remark following Theorem 2.7.3 in \cite{binghamGoldie}, ``for some $\delta > 0$'' can be replaced with ``for all $\delta > 0$'' in Definition \ref{def:quasiMonotoneFunction}.
\end{remark}

\begin{examples}
Classic examples of slowly varying functions at infinity include any function that converges to a positive real number at infinity. Moreover, Lemma 2.7.1 in \cite{binghamGoldie} argues that any monotone function which is slowly varying at infinity is quasi-monotone. In particular, the functions $x \mapsto \log(x)$, $x \mapsto 1/\log(x)$, and $x \mapsto \log(\log(x))$ are slowly varying at infinity and quasi-monotone.
\end{examples}

We now have all the necessary tools to introduce more satisfying definitions of long memory. Consider a second-order stationary real-valued time series $(X_{t})_{t \in \mathbb{Z}}$. The following conditions of long memory are essentially drawn from \cite{pipirasTaqqu2017} (Section 2.1), \cite{giraitisKoulSurgailis} (Section 3.1), \cite{beran2013} (Section 1.3.1), and \cite{samorodnitsky2016} (Section 6.2), although we have made a few adjustments for the sake of mathematical rigor.

The sequence $(X_{t})$ is said to be long-range dependent if it satisfies one of the two conditions below.

\begin{condition}
\label{condition:longMemoryAutocovariances}
The autocovariance function $\gamma$ of the time series $(X_{t})$ satisfies $$\gamma(h)=L_{1}(h){h}^{2d-1},\,\forall h \in \mathbb{N}^{*},$$ where $d \in (0,1/2)$ and $L_{1} \colon [0,+\infty) \to \mathbb{R}$ is a slowly varying function at infinity.
\end{condition}

\begin{condition}
\label{condition:longMemorySpectralDensity}
The spectral measure of $(X_{t})$ is absolutely continuous with respect to the Lebesgue measure, and there is a representative $f$ of the equivalence class of the spectral density which satisfies $$f(\lambda)=L_{2}(\lambda)\lambda^{-2d},\,\forall \lambda \in (0,\pi],$$ where $d \in (0,1/2)$ and $L_{2} \colon (0,\pi] \to [0,+\infty]$ is a non-negative slowly varying function at zero.
\end{condition}

If one of these conditions is satisfied, the parameter $d \in (0,1/2)$ is called the memory parameter (\cite{giraitisKoulSurgailis}, Definition 3.1.3) or the long-range dependence parameter (\cite{pipirasTaqqu2017}, Definition 2.1.5).

\begin{remark}
Notice that Condition \ref{condition:longMemorySpectralDensity} can be extended so as to require the spectral measure to have a density with respect to the Lebesgue measure only on a neighbourhood of 0, say $(-\delta,\delta)$ with $\delta \in (0,\pi]$. This is for example the approach of Samorodnitsky in \cite{samorodnitsky2007} (Section 5) and \cite{samorodnitsky2016} (Section 6.2).

Moreover, some authors define short and long memory in the spectral domain similarly to Condition \ref{condition:longMemorySpectralDensity}, but writing $f$ as the sum of the trigonometric series with autocovariances as coefficients, without making any further assumptions on the underlying process. Such a definition is not satisfying, because a second-order stationary process with an absolutely continuous spectral measure may have a spectral density whose Fourier series diverges almost everywhere, cf Section \ref{section:hardyRogosinski}. This is why in Condition \ref{condition:longMemorySpectralDensity}, we chose to work with a representative of the spectral density, without referring to its Fourier~series.

In fact, even when there is a representative of the spectral density which is continuous, the definition of long memory involving Fourier series is probably not satisfying. Indeed, according to \cite{katznelson1966}, for any Borel set with zero Lebesgue measure, we can find a continuous complex function $f$ whose Fourier series diverges on this Borel set. Thus, it may be possible---though yet to be proven---to find a second-order stationary real-valued sequence with an absolutely continuous spectral measure and a continuous spectral density, such that the Fourier series of the spectral density diverges on a dense set of zero measure, like $\mathbb{Q} \cap \left[-\pi,\pi\right]$. Such a process would exhibit \textit{short memory in the spectral domain}, in the sense that it has a spectral density that is continuous at the origin (see \cite{samorodnitsky2016}, Remark 6.2.4).
\end{remark}

It is easy to see that Conditions \ref{condition:longMemoryAutocovariances} and \ref{condition:longMemorySpectralDensity} imply that the autocovariances of the sequence $(X_{t})_{t \in \mathbb{Z}}$ are not absolutely summable. Nonetheless, these two conditions are not necessarily equivalent a priori.

In addition, it is known that Condition \ref{condition:longMemoryAutocovariances} does not imply that the spectral measure of $(X_{t})_{t \in \mathbb{Z}}$ is absolutely continuous. For instance, \cite{gubner2005} gives an example of autocovariance function satisfying Condition \ref{condition:longMemoryAutocovariances} with $L_{1}(h) \sim c$ as $h \to +\infty$, where $c \in (0,+\infty)$, such that the associated spectral measure is singular. Even if we assume that the spectral measure is absolutely continuous, it is not obvious whether Condition \ref{condition:longMemoryAutocovariances} or Condition \ref{condition:longMemorySpectralDensity} implies that the Fourier series of the spectral density converges almost everywhere. Fortunately, if we are willing to make additional assumptions on the slowly varying functions, the notion of quasi-monotonicity guarantees that all these implications hold. Furthermore, the Fourier series of the spectral density will converge everywhere except at 0, and Condition \ref{condition:longMemoryAutocovariances} and Condition \ref{condition:longMemorySpectralDensity} become equivalent, as long as the slowly varying functions $L_{1}$ and $L_{2}$ are quasi-monotone.

\begin{proposition}
\label{prop:equivalenceL1L2QuasiMonotone}
The two following assertions are equivalent:
\vspace{-\parskip}
\begin{enumerate}[label=(\roman*)]
\item Condition \ref{condition:longMemoryAutocovariances} is satisfied with $L_{1}$ quasi-monotone,
\item Condition \ref{condition:longMemorySpectralDensity} is satisfied with $L_{2}$ quasi-monotone.
\end{enumerate}
\vspace{-\parskip}
In addition, if one of these assertions is satisfied, the spectral measure of $(X_{t})_{t \in \mathbb{Z}}$ is absolutely continuous, and the Fourier series of its spectral density converges everywhere on $(0,\pi]$ and diverges to $+\infty$ at 0. Moreover, Condition \ref{condition:longMemorySpectralDensity} holds with $f$ the sum of the Fourier series of the spectral density, and we have $$L_{2}(x) \underset{x \to 0}{\sim} L_{1}\hspace{-1.6pt}\left(\frac{1}{x}\right)\frac{1}{\pi}\Gamma(2d)\cos(d\pi),\,\text{and}~L_{1}(x) \underset{x \to +\infty}{\sim} L_{2}\hspace{-1.6pt}\left(\frac{1}{x}\right)2\Gamma(1-2d)\sin(d\pi).$$
\end{proposition}

\begin{proof}[Proposition \ref{prop:equivalenceL1L2QuasiMonotone}]
This follows from Proposition 2.2.14 and Proposition A.2.2 in \cite{pipirasTaqqu2017}.
\end{proof}

\begin{remark}
\cite{giraitisKoulSurgailis} (Proposition 3.1.1), \cite{beran2013} (Theorem 1.3), and \cite{samorodnitsky2016} (Theorem 6.2.11), present a similar equivalence between Conditions \ref{condition:longMemoryAutocovariances} and \ref{condition:longMemorySpectralDensity}, but under the stronger assumption that $L_{1}$ and $L_{2}$ are locally of bounded variation and belong to the Zygmund class. The Zygmund class is the class of measurable functions $L$ that are positive on a neighbourhood of infinity and such that for all $\delta > 0$, $x \mapsto x^{\delta}L(x)$ is ultimately increasing and $x \mapsto x^{-\delta}L(x)$ is ultimately decreasing. By Theorem 1.5.5 in \cite{binghamGoldie}, the Zygmund class coincides with the class of normalised slowly varying functions, i.e., the class of functions $L$ which, for $x$ large enough, can be represented as
\begin{equation}
\label{eq:zygmundClassRepresentation}
L(x)=c\exp\hspace{-1.6pt}\left(\int_{b}^{x}\frac{\varepsilon(u)}{u}du\right)\hspace{-1.7pt},
\end{equation}
for some positive constants $b$ and $c$ and some measurable function $\varepsilon$ converging to 0 at infinity. Thanks to Theorem 1.5.5 and Corollary 2.7.4 in \cite{binghamGoldie}, the set of functions which are locally of bounded variation and belong to the Zygmund class is a proper subset of the class of quasi-monotone slowly varying functions.

Finally, a single asymptotic condition on the spectral density near zero is not sufficient to ensure that Condition \ref{condition:longMemorySpectralDensity} implies Condition \ref{condition:longMemoryAutocovariances}. We need an additional condition to guarantee that the spectral density is well-behaved everywhere except at 0, and this is why we assume that it is locally of bounded variation. For instance, in \cite{samorodnitsky2007}, Example 5.2 shows a process which satisfies Condition \ref{condition:longMemorySpectralDensity} but not Condition \ref{condition:longMemoryAutocovariances}, by allowing the spectral density to explode at the point $\pi$.
\end{remark}

\begin{remark}
According to Taqqu, the most general known condition to relate Conditions \ref{condition:longMemoryAutocovariances} and \ref{condition:longMemorySpectralDensity} is quasi-monotonicity. Still, to the best of our knowledge, there are no results proving that quasi-monotonicity is a necessary condition so that the conclusion of Proposition \ref{prop:equivalenceL1L2QuasiMonotone} holds. It may be possible to find weaker conditions than quasi-monotonicity ensuring that the spectral measure is absolutely continuous with an almost everywhere convergent Fourier series, while also making Conditions \ref{condition:longMemoryAutocovariances} and \ref{condition:longMemorySpectralDensity} equivalent.

In fact, if we set aside the connection between Conditions \ref{condition:longMemoryAutocovariances} and \ref{condition:longMemorySpectralDensity}, we can find weaker conditions on $L_{1}$ under which Condition \ref{condition:longMemoryAutocovariances} implies that the spectral measure is absolutely continuous, and the spectral density admits a Fourier representation everywhere except at 0. For instance, according to Theorem 6.2.11 in \cite{samorodnitsky2016}, if Condition \ref{condition:longMemoryAutocovariances} is satisfied with $L_{1}$ in the Zygmund class but not necessarily locally of bounded variation, then Condition \ref{condition:longMemorySpectralDensity} also holds, along with the last part of Proposition \ref{prop:equivalenceL1L2QuasiMonotone} regarding the Fourier series of the spectral density and the asymptotic equalities for both $L_{1}$ and $L_{2}$.

It may be possible to find even weaker conditions to ensure that the spectral density can be written as the sum of its Fourier series. Assuming that $L_{1}$ belongs to the Zygmund class means that it can be represented as in (\ref{eq:zygmundClassRepresentation}). In fact, all slowly varying functions admit such a representation, with the constant $c$ replaced by a measurable function converging at infinity to a positive real number (see \cite{binghamGoldie}, Theorem 1.3.1). Instead of taking $c$ to be constant in this representation, as in the Zygmund class, one could relax this assumption and allow $c$ to vary subject to appropriate asymptotic conditions. For instance, if we assume that Condition \ref{condition:longMemoryAutocovariances} is satisfied with $c$ such that $c(h+1) / c(h)=1+O\hspace{-1.6pt}\left(\frac{1}{h^{\delta}}\right)$, where $\delta \in (2d,1)$, we get $h(\gamma(h)-\gamma(h+1))=o(1)$. In \cite{stanojevic1992}, Stanojevic showed that this implies that the trigonometric series $\sum_{h \in \mathbb{Z}}\gamma(h)\e^{\mathrm{i}h\lambda}$ converges almost everywhere. However, further exploration---or potentially additional assumptions---may be required to ensure that the spectral density is equal to the sum of this series almost everywhere.

Similarly, stronger assumptions can be made on Condition \ref{condition:longMemorySpectralDensity} to guarantee that it implies that the Fourier series of the spectral density converges almost everywhere. For example, \cite{giraitisKoulSurgailis} (Definition 3.1.4) defines long-range dependence in the spectral domain using the same statement as Condition \ref{condition:longMemorySpectralDensity}, but with the additional requirement that the representative $f$ is locally bounded on $(0,\pi]$. This forces $L_{2}$ to be locally bounded on $(0,\pi]$. Then, choosing $p$ and $\delta$ such that $1 < p < \frac{1}{2d}$ and $2dp < \delta < 1$, we get $f^{p}(\lambda)=o\hspace{-1.6pt}\left(\rule{0pt}{9.1pt}\right.\hspace{-3.2pt}\lambda^{-\delta}\hspace{-3.2pt}\left.\rule{0pt}{9.1pt}\right)$ as $\lambda \to 0$. Since $\delta < 1$, it follows that $f$ is in $L^{p}$, with $p > 1$. By the Carleson--Hunt theorem \cite{carleson1966, hunt1967}, we can conclude that the Fourier series of $f$ converges almost~everywhere~to~$f$.
\end{remark}

Finally, we have restricted our attention to second-order definitions of long memory. Other approaches have been proposed that do not necessarily require second-order moments and therefore fall outside the scope of this article. For instance, in \cite{samorodnitsky2016}, Section 5.3, Samorodnitsky argues that ergodicity and mixing are natural concepts for describing the memory of a stationary process, although he points out that ``unusual'' phenomena have been observed in mixing Gaussian processes whose autocovariances decrease slowly enough. Moreover, as mentioned in \cite{giraitisKoulSurgailis}, Section 3.1, the $\alpha$-mixing condition is not easy to verify and may appear too restrictive compared to definitions based on the autocovariances or the spectral density of the process.

Another approach, known as distributional long memory, classifies a process as long-memory if its partial sum of the first $\lfloor{n\tau}\rfloor$ terms, normalized by a proper constant (typically a power of $n$), converges in the sense of finite-dimensional distributions to a continuous-time process $J({\tau})$ whose increments are not independent. The concept of distributional long memory originates in Cox \cite{cox1984}, and was formalized by Dehling and Philipp \cite{dehling2002}. Unlike second-order definitions, this approach does not require the existence of moments, making it applicable to processes with infinite variance \cite{leipusRecentDevelopments}.

In 2016, Samorodnitsky introduced the notion of phase transition to determine when a process can be considered long-range dependent \cite{samorodnitsky2016}. In this framework, long memory is expressed as a deviation from the behavior of an i.i.d. sequence with the same (or similar) marginal distribution as the stationary process under consideration, measured through some functional $\phi$ (for instance the sequence of partial sums or the sequence of partial maxima).

\section*{Acknowledgements}

The author is grateful to their PhD supervisors, James Ridgway (Capital Fund Management) and Pierre Alquier (ESSEC Business School), for their guidance and support. This work was supported by Lux Horizon Technologies.

\newpage

\printbibliography

\appendix

\section{Appendix: Technical Proofs}
\label{section:technicalProofs}

\subsection{Proof of Lemma \ref{lemma:HardyTrigonometricPolynomials}}
\label{subsection:appendixHardy}

For the sake of readability, some lemmas used in the following proof are stated and proved at the end of Subsection \ref{subsection:appendixHardy}.

We start by introducing some notations that will be useful throughout the proof. Take $n \in \mathbb{N}$ such that $n \ge 50$. We put $x_{n}$ the integer defined as $$x_{n}=\left\lfloor\frac{4n+1}{4}\hspace{-1.6pt}\left(1-\frac{1}{{(\log(n))}^{\frac{1}{6}}}\right)\right\rfloor\hspace{-1.7pt}.$$ It can be easily shown that $n$ is big enough to ensure that $x_{n} \ge 2$.

For all $l=1,\dotsc,x_{n}$, we put $$A_{l}^{n}=\frac{4\pi{l}}{4n+1}.$$ We also consider some integers $m_{1}^{n},\dotsc,m_{x_{n}}^{n}$ such that $m_{1}^{n} \ge n^{4}$, $m_{l+1}^{n} > 2m_{l}^{n}$ for all $l=1,\dotsc,x_{n}-1$, and $2m_{l}^{n}+1$ is a multiple of $4n+1$ for all $l=1,\dotsc,x_{n}$.

Now, we define the polynomial kernel $f_{n}$ as an average of Fejér kernels: $$f_{n}(\theta)=\frac{1}{x_{n}}\sum_{l=1}^{x_{n}}F_{m_{l}^{n}}(\theta-A_{l}^{n}).$$ The function $f_{n}$ is clearly non-negative with constant term $1/2$. To obtain an even polynomial function, we then define $\varphi_{n}$ as $$\varphi_{n}(\theta)=\frac{f_{n}(\theta)+f_{n}(-\theta)}{2}.$$ The sequence $(\varphi_{n})_{n \in \mathbb{N}}$ is the one for which we will establish a lower bound on the modulus of the symmetric partial Fourier sums that diverges to $+\infty$.

Finally, for all $j=1,\dotsc,x_{n}-1$, we define the interval $J_{j}^{n}$ as $$J_{j}^{n}=\left]A_{j}^{n}+\frac{1}{n^{2}},A_{j+1}^{n}-\frac{1}{n^2}\right[.$$ Notice that, since $n \ge 1$, none of the intervals $J_{j}^{n}$s are empty. We then put $$\Omega_{n}=\bigcup_{j=1}^{\left\lfloor{x_{n}-\sqrt{n}}\right\rfloor}\left\{\theta \in J_{j}^{n}~\middle|~\theta \ge \frac{1}{{(\log(n))}^{\frac{1}{6}}}~\text{and}~\left|\sin\hspace{-1.6pt}\left(\hspace{-1.6pt}\left(m_{j}^{n}+\frac{1}{2}\right)\hspace{-1.6pt}\theta\right)\hspace{-1.6pt}\right| \ge \frac{1}{{(\log(n))}^{\frac{1}{6}}}\right\}\hspace{-1.7pt}.$$ Note that $n$ is large enough so that $\left\lfloor{x_{n}-\sqrt{n}}\right\rfloor \ge 1$. The additional conditions involving $(\log(n))^{1/6}$ in the definition of the set $\Omega_{n}$ are such that it may be empty for small $n$. However, as we shall see, for sufficiently large $n$, $\Omega_{n}$ becomes non-empty, and its measure even tends to $\pi$ as $n$ goes to infinity. This implies that $E_{n}=\Omega_{n} \cup (-\Omega_{n})$ asymptotically covers almost all of the interval $[-\pi,\pi]$. The lower bound on the modulus of the symmetric partial Fourier sums of $\varphi_{n}$ will be shown to hold for all $\theta \in E_{n}$, thus concluding the proof.

\begin{proof}[Lemma \ref{lemma:HardyTrigonometricPolynomials}]
The function $\varphi_{n}$ non-negative and even with constant term $1/2$. Moreover, it can be easily seen that $\varphi_{n}$ is of degree at most $m_{x_{n}}^{n}$, and the monomial term of degree $m_{x_{n}}^{n}$ of $\varphi_{n}$ evaluated at $\theta$ is equal to $$\frac{1}{2x_{n}}\left(1-\frac{m_{x_{n}}^{n}}{m_{x_{n}}^{n}+1}\right)\left(\cos(m_{x_{n}}^{n}(\theta-A_{x_{n}}^{n}))+\cos(m_{x_{n}}^{n}(\theta+A_{x_{n}}^{n}))\right)\hspace{-1.7pt},$$ which is also equal to $$\frac{1}{2x_{n}}\left(1-\frac{m_{x_{n}}^{n}}{m_{x_{n}}^{n}+1}\right)\cos(m_{x_{n}}^{n}A_{x_{n}}^{n})\left(\e^{\mathrm{i}m_{x_{n}}^{n}\theta}+\e^{-\mathrm{i}m_{x_{n}}^{n}\theta}\right)\hspace{-1.7pt}.$$ Hence, we see that $\varphi_{n}$ is a polynomial of degree $m_{x_{n}}^{n}$ if and only if $m_{x_{n}}^{n}A_{x_{n}}^{n} \not\equiv \frac{\pi}{2} \pmod{\pi}$. But since $2m_{x_{n}}^{n}+1$ is a multiple of $4n+1$, there exists $k \in \mathbb{N}^{*}$ such that $m_{x_{n}}^{n}=\frac{k}{2}(4n+1)-\frac{1}{2}$. We deduce that $$m_{x_{n}}^{n}A_{x_{n}}^{n}=\left(\frac{k}{2}(4n+1)-\frac{1}{2}\right)\frac{4\pi{x_{n}}}{4n+1}=2\pi{k}x_{n}-\frac{2\pi{x_{n}}}{4n+1}.$$ Thus, $m_{x_{n}}^{n}A_{x_{n}}^{n} \not\equiv \frac{\pi}{2} \pmod{\pi}$ if and only if $\frac{2\pi{x_{n}}}{4n+1} \not\equiv \frac{\pi}{2} \pmod{\pi}$. It is clear that $0 < x_{n} < \frac{4n+1}{4}$, so that $0 < \frac{2\pi{x_{n}}}{4n+1} < \frac{\pi}{2}$. This implies that $\frac{2\pi{x_{n}}}{4n+1} \not\equiv \frac{\pi}{2} \pmod{\pi}$, which shows that $\varphi_{n}$ has degree $m_{x_{n}}^{n}$.

Now, take $\theta \in \Omega_{n}$. There exists some $j \in \left\{1,\dotsc,\left\lfloor{x_{n}-\sqrt{n}}\right\rfloor\right\}$ such that $\theta \in J_{j}^{n}$ and
\begin{equation}
\label{eq:lowerBoundThetaAndSinMjTheta}
\theta \ge \frac{1}{{(\log(n))}^{\frac{1}{6}}},\,\text{and}~\left|\sin\hspace{-1.6pt}\left(\hspace{-1.6pt}\left(m_{j}^{n}+\frac{1}{2}\right)\hspace{-1.6pt}\theta\right)\hspace{-1.6pt}\right| \ge \frac{1}{{(\log(n))}^{\frac{1}{6}}}.
\end{equation}
We start by calculating $S_{m_{j}^{n}}(\varphi_{n},\theta)$. Define the function $g_{n}$ as $g_{n} \colon \theta \mapsto f_{n}(-\theta)$. For all $l \in \mathbb{N}$, we have $$\hat{g}_{n}(l)=\frac{1}{2\pi}\int_{-\pi}^{\pi}f_{n}(-\theta)\e^{-\mathrm{i}l\theta}d\theta=\frac{1}{2\pi}\int_{-\pi}^{\pi}f_{n}(\theta)\e^{-\mathrm{i}(-l)\theta}d\theta=\hat{f}_{n}(-l).$$ Using classical properties of Fourier coefficients, it follows that
\begin{flalign*}
S_{m_{j}^{n}}(\varphi_{n},\theta) &= \frac{1}{2}\sum_{l=-m_{j}^{n}}^{m_{j}^{n}}\left(\hat{f}_{n}(l)+\hat{g}_{n}(l)\right)\e^{\mathrm{i}l\theta} \\
&= \frac{1}{2}\sum_{l=-m_{j}^{n}}^{m_{j}^{n}}\hat{f}_{n}(l)\e^{\mathrm{i}l\theta}+\frac{1}{2}\sum_{l=-m_{j}^{n}}^{m_{j}^{n}}\hat{f}_{n}(-l)\e^{\mathrm{i}l\theta} \\
&= \frac{1}{2}\sum_{l=-m_{j}^{n}}^{m_{j}^{n}}\hat{f}_{n}(l)\e^{\mathrm{i}l\theta}+\frac{1}{2}\sum_{l=-m_{j}^{n}}^{m_{j}^{n}}\hat{f}_{n}(l)\e^{\mathrm{i}l(-\theta)} \\
&= \frac{1}{2}\left(S_{m_{j}^{n}}(f_{n},\theta)+S_{m_{j}^{n}}(f_{n},-\theta)\right)\hspace{-1.7pt}. \vphantom{\sum_{l=-m_{j}^{n}}^{m_{j}^{n}}\hat{f}_{n}(l)\e^{\mathrm{i}l\theta}}
\end{flalign*}
Since $\theta$ belongs to $J_{j}^{n}$, Lemma \ref{lemma:upperBoundPartialFourierSumFn}, combined with the fact that $|a+b| \ge |a|-|b|$ for any $a,b \in \mathbb{R}$, ensures that we have the following inequality, where $C$ has been defined in Lemma~\ref{lemma:propertiesFejerKernel}:
\begin{equation}
\label{eq:lowerBoundDirichletKernels}
\left|S_{m_{j}^{n}}(\varphi_{n},\theta)\right| \ge \frac{1}{2}\left|\frac{1}{x_{n}}\sum_{l=j+1}^{x_{n}}\frac{m_{l}^{n}-m_{j}^{n}}{m_{l}^{n}+1}D_{m_{j}^{n}}(\theta-A_{l}^{n})+\frac{1}{x_{n}}\sum_{l=j+1}^{x_{n}}\frac{m_{l}^{n}-m_{j}^{n}}{m_{l}^{n}+1}D_{m_{j}^{n}}(\theta+A_{l}^{n})\right|-C.
\end{equation}
Moreover, we have assumed that $2m_{j}^{n}+1$ is a multiple of $4n+1$. Hence, for all $l=1,\dotsc,x_{n}$, we obtain $$\left(m_{j}^{n}+\frac{1}{2}\right)\hspace{-1.6pt}A_{l}^{n}=\left(2m_{j}^{n}+1\right)\frac{1}{2}\frac{4\pi{l}}{4n+1}=2\pi{l}\frac{2m_{j}^{n}+1}{4n+1},$$ from which follows that $$\left(m_{j}^{n}+\frac{1}{2}\right)\hspace{-1.6pt}A_{l}^{n} \equiv 0 \pmod{2\pi}.$$ We want to use this last result to simplify the numerators of the Dirichlet kernels inside (\ref{eq:lowerBoundDirichletKernels}). However, we first need to check that the denominators of these Dirichlet kernels are not equal to zero. We have $\theta < A_{j+1}^{n} \le A_{l}^{n}$ because $l \ge j+1$, and $A_{l}^{n}-\theta \le A_{l}^{n} \le \pi$. This means that $-\pi < \frac{1}{2}(\theta-A_{l}^{n}) < 0$, so that $\sin\hspace{-1.6pt}\left(\frac{1}{2}(\theta-A_{l}^{n})\right) \ne 0$. We also clearly have $0 < \theta < A_{j+1}^{n} \le A_{x_{n}}^{n} \le \pi$, and $0 < \frac{1}{2}(\theta+A_{l}^{n}) < \pi$, so that $\sin\hspace{-1.6pt}\left(\frac{1}{2}(\theta+A_{l}^{n})\right) \ne 0$. Thus, the sine function being $2\pi$-periodic, we can rewrite the Dirichlet kernels as $$D_{m_{j}^{n}}(\theta-A_{l}^{n})=\frac{\sin\hspace{-1.6pt}\left(\left(m_{j}^{n}+\frac{1}{2}\right)(\theta-A_{l}^{n})\right)}{2\sin\hspace{-1.6pt}\left(\frac{1}{2}(\theta-A_{l}^{n})\right)}=\frac{\sin\hspace{-1.6pt}\left(\left(m_{j}^{n}+\frac{1}{2}\right)\theta\right)}{2\sin\hspace{-1.6pt}\left(\frac{1}{2}(\theta-A_{l}^{n})\right)},$$ and $$D_{m_{j}^{n}}(\theta+A_{l}^{n})=\frac{\sin\hspace{-1.6pt}\left(\left(m_{j}^{n}+\frac{1}{2}\right)(\theta+A_{l}^{n})\right)}{2\sin\hspace{-1.6pt}\left(\frac{1}{2}(\theta+A_{l}^{n})\right)}=\frac{\sin\hspace{-1.6pt}\left(\left(m_{j}^{n}+\frac{1}{2}\right)\theta\right)}{2\sin\hspace{-1.6pt}\left(\frac{1}{2}(\theta+A_{l}^{n})\right)}.$$ Plugging these values of the Dirichlet kernels inside (\ref{eq:lowerBoundDirichletKernels}), we deduce that
\begin{flalign}
\left|S_{m_{j}^{n}}(\varphi_{n},\theta)\right| & \ge \frac{1}{4}\left|\sin\hspace{-1.6pt}\left(\hspace{-1.6pt}\left(m_{j}^{n}+\frac{1}{2}\right)\hspace{-1.6pt}\theta\right)\hspace{-1.6pt}\right|\left|\frac{1}{x_{n}}\sum_{l=j+1}^{x_{n}}\frac{m_{l}^{n}-m_{j}^{n}}{m_{l}^{n}+1}\hspace{-1.6pt}\left(\frac{1}{\sin\hspace{-1.6pt}\left(\frac{1}{2}(\theta-A_{l}^{n})\right)}+\frac{1}{\sin\hspace{-1.6pt}\left(\frac{1}{2}(\theta+A_{l}^{n})\right)}\right)\hspace{-1.6pt}\right|-C \nonumber \\
&= \frac{1}{4}\left|\sin\hspace{-1.6pt}\left(\hspace{-1.6pt}\left(m_{j}^{n}+\frac{1}{2}\right)\hspace{-1.6pt}\theta\right)\hspace{-1.6pt}\right|\left|\frac{1}{x_{n}}\sum_{l=j+1}^{x_{n}}\frac{m_{l}^{n}-m_{j}^{n}}{m_{l}^{n}+1}\hspace{-1.6pt}\left(\frac{1}{\sin\hspace{-1.6pt}\left(\frac{1}{2}(A_{l}^{n}-\theta)\right)}-\frac{1}{\sin\hspace{-1.6pt}\left(\frac{1}{2}(A_{l}^{n}+\theta)\right)}\right)\hspace{-1.6pt}\right|-C. \label{eq:lowerBoundAbsoluteSinus}
\end{flalign}
A sum-to-product identity tells us that $\sin(a)-\sin(b)=2\cos\hspace{-1.6pt}\left(\frac{a+b}{2}\right)\sin\hspace{-1.6pt}\left(\frac{a-b}{2}\right)$, for all $a,b$. We thus~have
\begin{equation}
\label{eq:sinSumToProduct}
\sin\hspace{-1.6pt}\left(\frac{1}{2}(A_{l}^{n}+\theta)\right)-\sin\hspace{-1.6pt}\left(\frac{1}{2}(A_{l}^{n}-\theta)\right)=2\cos\hspace{-1.6pt}\left(\frac{A_{l}^{n}}{2}\right)\sin\hspace{-1.6pt}\left(\frac{\theta}{2}\right)\hspace{-1.7pt},
\end{equation}
for all $l=j+1,\dotsc,x_{n}$. Since $0 \le A_{l}^{n} \le A_{x_{n}}^{n} \le \pi$, we have $0 \le \frac{A_{l}^{n}}{2} \le \frac{\pi}{2}$, which implies that $\cos\hspace{-1.6pt}\left(\frac{A_{l}^{n}}{2}\right) \ge 0$. In addition, $\theta \in J_{j}^{n}$ so that $0 \le A_{j}^{n} \le \theta \le A_{j+1}^{n} \le A_{x_{n}}^{n} \le \pi$. Hence, we have $0 \le \frac{\theta}{2} \le \frac{\pi}{2}$, which implies that $\sin\hspace{-1.6pt}\left(\frac{\theta}{2}\right) \ge 0$. Thus, we deduce from equality (\ref{eq:sinSumToProduct}) that for all $l=j+1,\dotsc,x_{n}$, we have the inequality
\begin{equation}
\label{eq:sinusInequality}
0 < \sin\hspace{-1.6pt}\left(\frac{1}{2}(A_{l}^{n}-\theta)\right) \le \sin\hspace{-1.6pt}\left(\frac{1}{2}(A_{l}^{n}+\theta)\right)\hspace{-1.7pt}.
\end{equation}
Indeed, we have $0 < \frac{1}{2}(A_{l}^{n}-\theta) \le \frac{\pi}{2}$, so that $\sin\hspace{-1.6pt}\left(\frac{1}{2}(A_{l}^{n}-\theta)\right) > 0$.

As a consequence of inequality (\ref{eq:sinusInequality}), which holds for all $l=j+1,\dotsc,x_{n}$, the terms inside the sum in the lower bound of inequality (\ref{eq:lowerBoundAbsoluteSinus}) are non-negative, and the whole sum is itself non-negative. We get
\vspace{-\parskip}
\begin{equation}
\label{eq:lowerBoundDifferenceOfSinus}
\left|S_{m_{j}^{n}}(\varphi_{n},\theta)\right| \ge \left|\sin\hspace{-1.6pt}\left(\hspace{-1.6pt}\left(m_{j}^{n}+\frac{1}{2}\right)\hspace{-1.6pt}\theta\right)\hspace{-1.6pt}\right|\frac{1}{4x_{n}}\sum_{l=j+1}^{x_{n}}\frac{m_{l}^{n}-m_{j}^{n}}{m_{l}^{n}+1}\hspace{-1.6pt}\left(\frac{1}{\sin\hspace{-1.6pt}\left(\frac{1}{2}(A_{l}^{n}-\theta)\right)}-\frac{1}{\sin\hspace{-1.6pt}\left(\frac{1}{2}(A_{l}^{n}+\theta)\right)}\right)-C.
\end{equation}
Now, put $l \in \left\{j+1,\dotsc,x_{n}\right\}$. We have
\begin{flalign}
\frac{1}{\sin\hspace{-1.6pt}\left(\frac{1}{2}(A_{l}^{n}-\theta)\right)}-\frac{1}{\sin\hspace{-1.6pt}\left(\frac{1}{2}(A_{l}^{n}+\theta)\right)} &= \frac{\sin\hspace{-1.6pt}\left(\frac{1}{2}(A_{l}^{n}+\theta)\right)-\sin\hspace{-1.6pt}\left(\frac{1}{2}(A_{l}^{n}-\theta)\right)}{\sin\hspace{-1.6pt}\left(\frac{1}{2}(A_{l}^{n}-\theta)\right)\sin\hspace{-1.6pt}\left(\frac{1}{2}(A_{l}^{n}+\theta)\right)} \nonumber \\
&= \frac{2\cos\hspace{-1.6pt}\left(\frac{A_{l}^{n}}{2}\right)\sin\hspace{-1.6pt}\left(\frac{\theta}{2}\right)}{\sin\hspace{-1.6pt}\left(\frac{1}{2}(A_{l}^{n}-\theta)\right)\sin\hspace{-1.6pt}\left(\frac{1}{2}(A_{l}^{n}+\theta)\right)}. \label{eq:equalityDifferenceSinus}
\end{flalign}
We have already seen that $\cos\hspace{-1.6pt}\left(\frac{A_{l}^{n}}{2}\right) \ge 0$, $\sin\hspace{-1.6pt}\left(\frac{\theta}{2}\right) \ge 0$, $\sin\hspace{-1.6pt}\left(\frac{1}{2}(A_{l}^{n}-\theta)\right) > 0$, and $\sin\hspace{-1.6pt}\left(\frac{1}{2}(A_{l}^{n}+\theta)\right) > 0$. Since we also have $\sin\hspace{-1.6pt}\left(\frac{1}{2}(A_{l}^{n}+\theta)\right) \le 1$, we deduce from equality (\ref{eq:equalityDifferenceSinus}) that
\begin{equation}
\label{eq:inequalityDifferenceSinus1}
\frac{1}{\sin\hspace{-1.6pt}\left(\frac{1}{2}(A_{l}^{n}-\theta)\right)}-\frac{1}{\sin\hspace{-1.6pt}\left(\frac{1}{2}(A_{l}^{n}+\theta)\right)} \ge \frac{2\cos\hspace{-1.6pt}\left(\frac{A_{l}^{n}}{2}\right)\sin\hspace{-1.6pt}\left(\frac{\theta}{2}\right)}{\sin\hspace{-1.6pt}\left(\frac{1}{2}(A_{l}^{n}-\theta)\right)}.
\end{equation}
Furthermore, notice that
\begin{equation}
\label{eq:inequalityAxn}
A_{x_{n}}^{n}=\frac{4\pi{x_{n}}}{4n+1}=\frac{4\pi}{4n+1}\left\lfloor\frac{4n+1}{4}\hspace{-1.6pt}\left(1-\frac{1}{{(\log(n))}^{\frac{1}{6}}}\right)\right\rfloor \le \pi\hspace{-1.6pt}\left(1-\frac{1}{{(\log(n))}^{\frac{1}{6}}}\right)\hspace{-1.7pt}.
\end{equation}
Because $A_{l}^{n} \le A_{x_{n}}^{n}$, inequality (\ref{eq:inequalityAxn}) implies that $$0 \le \frac{A_{l}^{n}}{2} \le \frac{\pi}{2}\hspace{-1.6pt}\left(1-\frac{1}{{(\log(n))}^{\frac{1}{6}}}\right) \le \frac{\pi}{2}.$$ The cosine function being non-increasing on $\left[0,\frac{\pi}{2}\right]$, it follows that
\begin{equation}
\label{eq:lowerBoundCosAl}
\cos\hspace{-1.6pt}\left(\frac{A_{l}^{n}}{2}\right) \ge \cos\hspace{-1.6pt}\left(\frac{\pi}{2}\hspace{-1.6pt}\left(1-\frac{1}{{(\log(n))}^{\frac{1}{6}}}\right)\right)\hspace{-1.7pt}.
\end{equation}
But for all $x \in \left[0,\frac{\pi}{2}\right]$, we have $\cos(x) \ge \frac{2}{\pi}\left(\frac{\pi}{2}-x\right)$, which, combined with inequality (\ref{eq:lowerBoundCosAl}), ensures~that $$\cos\hspace{-1.6pt}\left(\frac{A_{l}^{n}}{2}\right) \ge \frac{1}{{(\log(n))}^{\frac{1}{6}}}.$$ Consequently, inequality (\ref{eq:inequalityDifferenceSinus1}) becomes
\begin{equation}
\label{eq:inequalityDifferenceSinus2}
\frac{1}{\sin\hspace{-1.6pt}\left(\frac{1}{2}(A_{l}^{n}-\theta)\right)}-\frac{1}{\sin\hspace{-1.6pt}\left(\frac{1}{2}(A_{l}^{n}+\theta)\right)} \ge \frac{2}{{(\log(n))}^{\frac{1}{6}}}\frac{\sin\hspace{-1.6pt}\left(\frac{\theta}{2}\right)}{\sin\hspace{-1.6pt}\left(\frac{1}{2}(A_{l}^{n}-\theta)\right)}.
\end{equation}
Additionally, as per condition (\ref{eq:lowerBoundThetaAndSinMjTheta}), we have chosen $\theta$ such that $$0 \le \frac{1}{2{(\log(n))}^{\frac{1}{6}}} \le \frac{\theta}{2} \le \frac{\pi}{2}.$$ The sine function being non-decreasing on $\left[0,\frac{\pi}{2}\right]$, it follows that
\begin{equation}
\label{eq:lowerBoundSinTheta}
\sin\hspace{-1.6pt}\left(\frac{\theta}{2}\right) \ge \sin\hspace{-1.6pt}\left(\frac{1}{2{(\log(n))}^{\frac{1}{6}}}\right)\hspace{-1.7pt}.
\end{equation}
But for all $x \in \left[0,\frac{\pi}{2}\right]$, we have $\sin(x) \ge \frac{2}{\pi}x$, which, combined with inequality (\ref{eq:lowerBoundSinTheta}), ensures that $$\sin\hspace{-1.6pt}\left(\frac{\theta}{2}\right) \ge \frac{1}{\pi{(\log(n))}^{\frac{1}{6}}}.$$ Consequently, inequality (\ref{eq:inequalityDifferenceSinus2}) becomes
\begin{equation}
\label{eq:inequalityDifferenceSinus3}
\frac{1}{\sin\hspace{-1.6pt}\left(\frac{1}{2}(A_{l}^{n}-\theta)\right)}-\frac{1}{\sin\hspace{-1.6pt}\left(\frac{1}{2}(A_{l}^{n}+\theta)\right)} \ge \frac{2}{\pi}\frac{1}{{(\log(n))}^{\frac{2}{6}}}\frac{1}{\sin\hspace{-1.6pt}\left(\frac{1}{2}(A_{l}^{n}-\theta)\right)}.
\end{equation}
Finally, for all $x \ge 0$, $\sin(x) \le x$. Since $0 < \frac{1}{2}(A_{l}^{n}-\theta) \le \frac{\pi}{2}$ and $\frac{1}{2}(A_{l}^{n}-\theta) \le \frac{1}{2}(A_{l}^{n}-A_{j}^{n})$, we obtain $$\frac{1}{\sin\hspace{-1.6pt}\left(\frac{1}{2}(A_{l}^{n}-\theta)\right)} \ge \frac{1}{\frac{1}{2}(A_{l}^{n}-\theta)} \ge \frac{1}{\frac{1}{2}(A_{l}^{n}-A_{j}^{n})}=\frac{4n+1}{2\pi}\frac{1}{l-j}.$$ As a result, we deduce from inequality (\ref{eq:inequalityDifferenceSinus3}) that
\begin{equation}
\label{eq:inequalityDifferenceSinus4}
\frac{1}{\sin\hspace{-1.6pt}\left(\frac{1}{2}(A_{l}^{n}-\theta)\right)}-\frac{1}{\sin\hspace{-1.6pt}\left(\frac{1}{2}(A_{l}^{n}+\theta)\right)} \ge \frac{4n+1}{\pi^{2}}\frac{1}{{(\log(n))}^{\frac{2}{6}}}\frac{1}{l-j}.
\end{equation}
Inequality (\ref{eq:inequalityDifferenceSinus4}) holds for all $l=j+1,\dotsc,x_{n}$, and plugging this lower bound into inequality (\ref{eq:lowerBoundDifferenceOfSinus}),\,we~get
\begin{equation}
\label{eq:lowerBoundWithoutSinus}
\left|S_{m_{j}^{n}}(\varphi_{n},\theta)\right| \ge \left|\sin\hspace{-1.6pt}\left(\hspace{-1.6pt}\left(m_{j}^{n}+\frac{1}{2}\right)\hspace{-1.6pt}\theta\right)\hspace{-1.6pt}\right|\frac{1}{4\pi^{2}}\frac{4n+1}{x_{n}}\frac{1}{{(\log(n))}^{\frac{2}{6}}}\sum_{l=j+1}^{x_{n}}\frac{m_{l}^{n}-m_{j}^{n}}{m_{l}^{n}+1}\frac{1}{l-j}-C.
\end{equation}
Moreover, for all $l=j+1,\dotsc,x_{n}$, $m_{l}^{n} > 2m_{l-1}^{n} \ge 2m_{j}^{n}$, so that $m_{l}^{n} \ge 2m_{j}^{n}+1$, and $$\frac{m_{l}^{n}-m_{j}^{n}}{m_{l}^{n}+1} \ge \frac{1}{2}.$$ Hence, inequality (\ref{eq:lowerBoundWithoutSinus}) becomes
\begin{equation}
\label{eq:lowerBoundWithoutMj}
\left|S_{m_{j}^{n}}(\varphi_{n},\theta)\right| \ge \left|\sin\hspace{-1.6pt}\left(\hspace{-1.6pt}\left(m_{j}^{n}+\frac{1}{2}\right)\hspace{-1.6pt}\theta\right)\hspace{-1.6pt}\right|\frac{1}{8\pi^{2}}\frac{4n+1}{x_{n}}\frac{1}{{(\log(n))}^{\frac{2}{6}}}\sum_{l=j+1}^{x_{n}}\frac{1}{l-j}-C.
\end{equation}
But we have chosen $j$ such that $j \le x_{n}-\sqrt{n}$, so that $x_{n}-j \ge \sqrt{n}$, and we get
\begin{equation}
\label{eq:lowerBoundHarmonicSum}
\sum_{l=j+1}^{x_{n}}\frac{1}{l-j}=\sum_{l=1}^{x_{n}-j}\frac{1}{l} \ge \log(x_{n}-j) \ge \frac{1}{2}\log(n).
\end{equation}
Combined with the fact that $\theta$ satisfies condition (\ref{eq:lowerBoundThetaAndSinMjTheta}), inequalities (\ref{eq:lowerBoundWithoutMj}) and (\ref{eq:lowerBoundHarmonicSum}) yield
\begin{equation}
\label{eq:lowerBoundWithoutSum}
\left|S_{m_{j}^{n}}(\varphi_{n},\theta)\right| \ge \frac{1}{16\pi^{2}}\frac{4n+1}{x_{n}}{(\log(n))}^{\frac{1}{2}}-C.
\end{equation}
Denoting $$M_{n}=\frac{1}{16\pi^{2}}\frac{4n+1}{x_{n}}{(\log(n))}^{\frac{1}{2}}-C,$$ inequality (\ref{eq:lowerBoundWithoutSum}) tells us that $\left|S_{m_{j}^{n}}(\varphi_{n},\theta)\right| \ge M_{n}$. It can be easily shown that $x_{n} \sim (4n+1)/4$ as $n \to +\infty$, from which follows that that $M_{n} \rightarrow +\infty$ as $n$ goes to $+\infty$.

All in all, we have shown that for all $\theta \in \Omega_{n}$, we have $\left|S_{m_{j}^{n}}(\varphi_{n},\theta)\right| \ge M_{n}$, where $j$ is such that $\theta$ belongs to $J_{j}^{n}$. However, $\varphi_{n}$ is an even function, which implies that $S_{m_{j}^{n}}(\varphi_{n},-\theta)=S_{m_{j}^{n}}(\varphi_{n},\theta)$ for all $\theta \in [0,\pi]$. Hence, we can assert that for all $\theta \in \Omega_{n} \cup (-\Omega_{n})$, we have $\left|S_{m_{j}^{n}}(\varphi_{n},\theta)\right| \ge M_{n}$, where $j$ is such that $|\theta| \in J_{j}^{n}$. Thus, denoting $E_{n}=\Omega_{n} \cup (-\Omega_{n})$, it remains to show that $m(E_{n}) \rightarrow 2\pi$ as $n \to +\infty$. Since $\Omega_{n}$ and $-\Omega_{n}$ are disjoint and have same measure, it is sufficient to show that $m(\Omega_{n}) \rightarrow \pi$ as $n \to +\infty$.

First, notice that
\begin{equation}
\label{eq:equalityOmeganSigman}
\Omega_{n}=\left[\frac{1}{{(\log(n))}^{\frac{1}{6}}},\pi\right] \cap \Sigma_{n},
\end{equation}
where $$\Sigma_{n}=\bigcup_{j=1}^{\left\lfloor{x_{n}-\sqrt{n}}\right\rfloor}\left\{x \in J_{j}^{n}~\middle|~\left|\sin\hspace{-1.6pt}\left(\hspace{-1.6pt}\left(m_{j}^{n}+\frac{1}{2}\right)\hspace{-1.6pt}x\right)\hspace{-1.6pt}\right| \ge \frac{1}{{(\log(n))}^{\frac{1}{6}}}\right\}\hspace{-1.7pt}.$$ Since the $J_{j}^{n}$s are pairwise disjoint, we have 
\begin{flalign}
m(\Sigma_{n}) &= \sum_{j=1}^{\left\lfloor{x_{n}-\sqrt{n}}\right\rfloor}m\hspace{-1.6pt}\left(\left\{x \in J_{j}^{n}~\middle|~\left|\sin\hspace{-1.6pt}\left(\hspace{-1.6pt}\left(m_{j}^{n}+\frac{1}{2}\right)\hspace{-1.6pt}x\right)\hspace{-1.6pt}\right| \ge \frac{1}{{(\log(n))}^{\frac{1}{6}}}\right\}\right) \nonumber \\
&= \sum_{j=1}^{\left\lfloor{x_{n}-\sqrt{n}}\right\rfloor}\frac{1}{m_{j}^{n}+\frac{1}{2}}m\hspace{-1.6pt}\left(\Sigma_{n}^{(j)}\right)\hspace{-1.7pt}, \label{eq:equalitySigmanSumSigmanj}
\end{flalign}
where $$\Sigma_{n}^{(j)}=\left\{y \in \left]\left(m_{j}^{n}+\frac{1}{2}\right)\left(A_{j}^{n}+\frac{1}{n^{2}}\right)\hspace{-1.7pt},\left(m_{j}^{n}+\frac{1}{2}\right)\left(A_{j+1}^{n}-\frac{1}{n^2}\right)\right[~\middle|~\left|\sin(y)\right| \ge \frac{1}{{(\log(n))}^{\frac{1}{6}}}\right\}\hspace{-1.7pt}.$$ Now, denoting $$\tilde{\Sigma}_{n}^{(j,\mathsf{c})}=\left\{y \in \left]\left(m_{j}^{n}+\frac{1}{2}\right)A_{j}^{n},\left(m_{j}^{n}+\frac{1}{2}\right)A_{j+1}^{n}\right[~\middle|~\left|\sin(y)\right| < \frac{1}{{(\log(n))}^{\frac{1}{6}}}\right\}\hspace{-1.7pt},$$ we get $\Sigma_{n}^{(j,\mathsf{c})} \subset \tilde{\Sigma}_{n}^{(j,\mathsf{c})}$, where $\Sigma_{n}^{(j,\mathsf{c})}=\left(m_{j}^{n}+\frac{1}{2}\right)\hspace{-1.6pt}J_{j}^{n} \setminus \Sigma_{n}^{(j)}$ is the relative complement of $\Sigma_{n}^{(j)}$ in $\left(m_{j}^{n}+\frac{1}{2}\right)\hspace{-1.6pt}J_{j}^{n}$.

In addition, we have assumed that $2m_{j}^{n}+1$ is a multiple of $4n+1$, so that there is an integer $k \in \mathbb{N}^{*}$ such that $2m_{j}^{n}+1=(4n+1)k$. We then have $$\left]\left(m_{j}^{n}+\frac{1}{2}\right)A_{j}^{n},\left(m_{j}^{n}+\frac{1}{2}\right)A_{j+1}^{n}\right[=\left]2kj\pi,2kj\pi+2k\pi\right[,$$ from which we deduce that
\begin{flalign}
m\hspace{-1.6pt}\left(\tilde{\Sigma}_{n}^{(j,\mathsf{c})}\right) &= \sum_{l=0}^{2k-1}m\hspace{-1.6pt}\left(\left\{y \in \left]2kj\pi+l\pi,2kj\pi+(l+1)\pi\right[~\middle|~\left|\sin(y)\right| < \frac{1}{{(\log(n))}^{\frac{1}{6}}}\right\}\right) \nonumber \\
&= 2k\,m\hspace{-1.6pt}\left(\left\{y \in \left]0,\pi\right[~\middle|~\sin(y) < \frac{1}{{(\log(n))}^{\frac{1}{6}}}\right\}\right) \label{eq:equalityMeasureTildeSigman2Pi} \\
&= 4k\,m\hspace{-1.6pt}\left(\left\{y \in \left]0,\frac{\pi}{2}\right[~\middle|~\sin(y) < \frac{1}{{(\log(n))}^{\frac{1}{6}}}\right\}\right)\hspace{-1.7pt}. \label{eq:equalityMeasureTildeSigmanPiDivided2}
\end{flalign}
Note that equality (\ref{eq:equalityMeasureTildeSigman2Pi}) is a direct consequence of the $\pi$-periodicity of the function $y \mapsto \left|\sin(y)\right|$, while equality (\ref{eq:equalityMeasureTildeSigmanPiDivided2}) follows from the fact that the function $y \mapsto \left|\sin(y)\right|$ is both $\pi$-periodic and even.

Furthermore, since $n \ge 3$, we have ${(\log(n))}^{-\frac{1}{6}} \le 1$, and
\begin{flalign}
m\hspace{-1.6pt}\left(\left\{y \in \left]0,\frac{\pi}{2}\right[~\middle|~\sin(y) < \frac{1}{{(\log(n))}^{\frac{1}{6}}}\right\}\right) &= m\hspace{-1.6pt}\left(\left\{y \in \left]0,\frac{\pi}{2}\right[~\middle|~y < \arcsin\hspace{-1.6pt}\left(\frac{1}{{(\log(n))}^{\frac{1}{6}}}\right)\right\}\right) \nonumber \\
&= \arcsin\hspace{-1.6pt}\left(\frac{1}{{(\log(n))}^{\frac{1}{6}}}\right)\hspace{-1.7pt}. \label{eq:equalityFromSinToArcsin}
\end{flalign}
But we have $\arcsin(x) \le \frac{\pi}{2}x$ for all $x \in [0,1]$, and equality (\ref{eq:equalityFromSinToArcsin}) implies that
\begin{equation}
\label{eq:upperBoundSetSin}
m\hspace{-1.6pt}\left(\left\{y \in \left]0,\frac{\pi}{2}\right[~\middle|~\sin(y) < \frac{1}{{(\log(n))}^{\frac{1}{6}}}\right\}\right) \le \frac{\pi}{2}\frac{1}{{(\log(n))}^{\frac{1}{6}}}.
\end{equation}
Inequality (\ref{eq:upperBoundSetSin}), together with equality (\ref{eq:equalityMeasureTildeSigmanPiDivided2}), yields $$m\hspace{-1.6pt}\left(\tilde{\Sigma}_{n}^{(j,\mathsf{c})}\right) \le \left(m_{j}^{n}+\frac{1}{2}\right)\frac{4\pi}{4n+1}\frac{1}{{(\log(n))}^{\frac{1}{6}}}.$$ Since $\Sigma_{n}^{(j,\mathsf{c})} \subset \tilde{\Sigma}_{n}^{(j,\mathsf{c})}$, we obtain
\begin{flalign*}
m\hspace{-1.6pt}\left(\Sigma_{n}^{(j)}\right) &= m\hspace{-1.6pt}\left(\left(m_{j}^{n}+\frac{1}{2}\right)\hspace{-1.6pt}J_{j}^{n}\right)-m\hspace{-1.6pt}\left(\Sigma_{n}^{(j,\mathsf{c})}\right) \vphantom{\left(\frac{4\pi}{4n+1}-\frac{2}{n^{2}}-\frac{4\pi}{4n+1}\frac{1}{{(\log(n))}^{\frac{1}{6}}}\right)} \\
& \ge m\hspace{-1.6pt}\left(\left(m_{j}^{n}+\frac{1}{2}\right)\hspace{-1.6pt}J_{j}^{n}\right)-m\hspace{-1.6pt}\left(\tilde{\Sigma}_{n}^{(j,\mathsf{c})}\right) \vphantom{\left(\frac{4\pi}{4n+1}-\frac{2}{n^{2}}-\frac{4\pi}{4n+1}\frac{1}{{(\log(n))}^{\frac{1}{6}}}\right)} \\
& \ge \left(m_{j}^{n}+\frac{1}{2}\right)\left(\frac{4\pi}{4n+1}-\frac{2}{n^{2}}-\frac{4\pi}{4n+1}\frac{1}{{(\log(n))}^{\frac{1}{6}}}\right)\hspace{-1.7pt}.
\end{flalign*}
Thus, it follows from equality (\ref{eq:equalitySigmanSumSigmanj}) that
\begin{equation}
\label{eq:squeezeMeasureSigman}
\left\lfloor{x_{n}-\sqrt{n}}\right\rfloor\left(\frac{4\pi}{4n+1}-\frac{2}{n^{2}}-\frac{4\pi}{4n+1}\frac{1}{{(\log(n))}^{\frac{1}{6}}}\right) \le m(\Sigma_{n}) \le \pi.
\end{equation}
Additionally, one can easily show that $\left\lfloor{x_{n}-\sqrt{n}}\right\rfloor \sim (4n+1)/4$ as $n \to +\infty$. Hence, the leftmost side of inequality (\ref{eq:squeezeMeasureSigman}) converges to $\pi$, and we get $m(\Sigma_{n}) \rightarrow \pi$ as $n \to +\infty$. From equality (\ref{eq:equalityOmeganSigman}), it is then straightforward to see that $m(\Omega_{n}) \rightarrow \pi$ as $n \to +\infty$.

In summary, for all $\theta \in E_{n}$, we have $\left|S_{p_{n}(\theta)}(\varphi_{n},\theta)\right| \ge M_{n}$, where $p_{n}(\theta)=m_{j}^{n}$ and $j$ is such that $|\theta| \in J_{j}^{n}$. Thus, we have $n \le p_{n}(\theta) \le q_{n}$ where $q_{n}=m_{x_{n}}^{n}$. We have also shown that $M_{n} \rightarrow +\infty$~and~$m(E_{n}) \rightarrow 2\pi$.

Finally, there exists an integer $n_{0} \ge 50$ such that for all $n \ge n_{0}$, we have $M_{n} > 0$. The conditions of Lemma \ref{lemma:HardyTrigonometricPolynomials} are then satisfied by the sequences $(\varphi_{n})_{n \ge n_{0}}$, $(M_{n})_{n \ge n_{0}}$, and $(E_{n})_{n \ge n_{0}}$. To obtain valid sequences with index set $\mathbb{N}$, we take $\varphi_{n}=\varphi_{n_{0}}$, $M_{n}=M_{n_{0}}$, and $E_{n}=E_{n_{0}}$, for every integer $n$ such that $n < n_{0}$. The proof is now complete.
\end{proof}

The following lemmas use the same notation as introduced at the beginning of this appendix.

\begin{lemma}
\label{lemma:FejerKernelsUniformlyBounded}
The Fejér kernels inside the formulas of $f_{n}(\theta)$ and $f_{n}(-\theta)$ are uniformly bounded inside the $J_{j}^{n}$s, in the sense that $$\forall \theta \in \bigcup_{j=1}^{x_{n}-1}J_{j}^{n},\,\forall k,l=1,\dotsc,x_{n},\,F_{m_{k}^{n}}(\theta-A_{l}^{n}) \le C~\text{and}~F_{m_{k}^{n}}(\theta+A_{l}^{n}) \le C,$$ where $C$ has been defined in Lemma \ref{lemma:propertiesFejerKernel}.
\end{lemma}

\begin{proof}[Lemma \ref{lemma:FejerKernelsUniformlyBounded}]
Put $j \in \left\{1,\dotsc,x_{n}-1\right\}$, take $\theta \in J_{j}^{n}$, and take $k,l \in \left\{1,\dotsc,x_{n}\right\}$. Then, we have $|\theta-A_{l}^{n}| \ge \frac{1}{n^{2}}$. Indeed,
\vspace{-\parskip}
\begin{itemize}[label=\upshape\textbullet]
\item if $l \le j$, then $A_{l}^{n} \le A_{j}^{n} < \theta-\frac{1}{n^{2}}$ since $\theta \in J_{j}^{n}$. Thus, $\theta-A_{l}^{n} > \frac{1}{n^{2}}$ and $|\theta-A_{l}^{n}| \ge \frac{1}{n^{2}}$,
\item if $l > j$, then $A_{l}^{n} \ge A_{j+1}^{n} > \theta+\frac{1}{n^{2}}$ since $\theta \in J_{j}^{n}$. Thus, $\theta-A_{l}^{n} < -\frac{1}{n^{2}}$ and $|\theta-A_{l}^{n}| \ge \frac{1}{n^{2}}$.
\end{itemize}
\vspace{-\parskip}
Moreover, we clearly have $0 < |\theta-A_{l}^{n}| \le \pi$. By Lemma \ref{lemma:propertiesFejerKernel} and because $m_{k}^{n} \ge m_{1}^{n} \ge n^{4}$, we obtain $$F_{m_{k}^{n}}(\theta-A_{l}^{n})=F_{m_{k}^{n}}(|\theta-A_{l}^{n}|) \le \frac{C}{(m_{k}^{n}+1){(\theta-A_{l}^{n})}^{2}} \le \frac{Cn^{4}}{m_{k}^{n}+1} \le C.$$ On the other hand, we have $\frac{1}{n^{2}} < \theta+A_{l}^{n} < 2\pi-\frac{1}{n^{2}}$. Indeed, $\theta+A_{l}^{n} > A_{j}^{n}+A_{l}^{n}+\frac{1}{n^{2}} \ge \frac{1}{n^{2}}$, and $\theta+A_{l}^{n} < A_{j+1}^{n}+A_{l}^{n}-\frac{1}{n^{2}} \le 2A_{{x}_{n}}^{n}-\frac{1}{n^{2}} \le 2\pi-\frac{1}{n^{2}}$. With the same reasoning as previously, it follows~that

\vspace{-\parskip}
\begin{itemize}[label=\upshape\textbullet]
\item if $0 < \theta+A_{l}^{n} \le \pi$, then by Lemma \ref{lemma:propertiesFejerKernel} and because $m_{k}^{n} \ge m_{1}^{n} \ge n^{4}$, we get $F_{m_{k}^{n}}(\theta+A_{l}^{n}) \le C$,
\item if $\pi < \theta+A_{l}^{n} < 2\pi$, then $0 < 2\pi-\theta-A_{l}^{n} \le \pi$. By Lemma \ref{lemma:propertiesFejerKernel} and because $m_{k}^{n} \ge n^{4}$, we get $$F_{m_{k}^{n}}(\theta+A_{l}^{n})=F_{m_{k}^{n}}(\theta+A_{l}^{n}-2\pi)=F_{m_{k}^{n}}(2\pi-\theta-A_{l}^{n}) \le \frac{C}{(m_{k}^{n}+1){(2\pi-\theta-A_{l}^{n})}^{2}} \le C.$$
\end{itemize}
In all cases, we get $F_{m_{k}^{n}}(\theta-A_{l}^{n}) \le C$ and $F_{m_{k}^{n}}(\theta+A_{l}^{n}) \le C$, which concludes the proof.
\end{proof}

\begin{lemma}
\label{lemma:equalityPartialFourierSumFn}
For all $j=1,\dotsc,x_{n}-1$ and for all $\theta \in \mathbb{R}$, we have $$S_{m_{j}^{n}}(f_{n},\theta)=\frac{1}{x_{n}}\sum_{l=1}^{j}F_{m_{l}^{n}}(\theta-A_{l}^{n})+\frac{1}{x_{n}}\sum_{l=j+1}^{x_{n}}\frac{m_{j}^{n}+1}{m_{l}^{n}+1}F_{m_{j}^{n}}(\theta-A_{l}^{n})+\frac{1}{x_{n}}\sum_{l=j+1}^{x_{n}}\frac{m_{l}^{n}-m_{j}^{n}}{m_{l}^{n}+1}D_{m_{j}^{n}}(\theta-A_{l}^{n}).$$
\end{lemma}

\begin{proof}[Lemma \ref{lemma:equalityPartialFourierSumFn}]
Since $f_{n}$ is a trigonometric polynomial, its $m_{j}^{n}$-th symmetric partial Fourier sum can be obtained by only keeping the monomial terms of $f_{n}$ that have degree smaller than or equal to $m_{j}^{n}$. Thus, we~have
\begin{flalign}
S_{m_{j}^{n}}(f_{n},\theta) &= \frac{1}{x_{n}}\left(\sum_{l=1}^{j}F_{m_{l}^{n}}(\theta-A_{l}^{n})+\sum_{l=j+1}^{x_{n}}\left(\frac{1}{2}\sum_{r=-m_{j}^{n}}^{m_{j}^{n}}\left(1-\frac{|r|}{m_{l}^{n}+1}\right)\e^{ir(\theta-A_{l}^{n})}\right)\right) \nonumber \\
&= \frac{1}{x_{n}}\left(\sum_{l=1}^{j}F_{m_{l}^{n}}(\theta-A_{l}^{n})+\sum_{l=j+1}^{x_{n}}\left(\frac{1}{2}+\sum_{r=1}^{m_{j}^{n}}\left(1-\frac{r}{m_{l}^{n}+1}\right)\cos(r(\theta-A_{l}^{n}))\right)\right)\hspace{-1.7pt}. \label{eq:equalitySmjFejer}
\end{flalign}
Together with the fact that $$1=\frac{m_{j}^{n}+1}{m_{l}^{n}+1}+\frac{m_{l}^{n}-m_{j}^{n}}{m_{l}^{n}+1},\,\text{and}~1-\frac{r}{m_{l}^{n}+1}=\frac{m_{j}^{n}+1}{m_{l}^{n}+1}\frac{m_{j}^{n}+1-r}{m_{j}^{n}+1}+\frac{m_{l}^{n}-m_{j}^{n}}{m_{l}^{n}+1},$$ equality (\ref{eq:equalitySmjFejer}) implies that
\begin{flalign*}
S_{m_{j}^{n}}(f_{n},\theta) &= \frac{1}{x_{n}}\left(\sum_{l=1}^{j}F_{m_{l}^{n}}(\theta-A_{l}^{n})+\sum_{l=j+1}^{x_{n}}\left(\frac{m_{j}^{n}+1}{m_{l}^{n}+1}F_{m_{j}^{n}}(\theta-A_{l}^{n})+\frac{m_{l}^{n}-m_{j}^{n}}{m_{l}^{n}+1}D_{m_{j}^{n}}(\theta-A_{l}^{n})\right)\right) \\
&= \frac{1}{x_{n}}\sum_{l=1}^{j}F_{m_{l}^{n}}(\theta-A_{l}^{n})+\frac{1}{x_{n}}\sum_{l=j+1}^{x_{n}}\frac{m_{j}^{n}+1}{m_{l}^{n}+1}F_{m_{j}^{n}}(\theta-A_{l}^{n})+\frac{1}{x_{n}}\sum_{l=j+1}^{x_{n}}\frac{m_{l}^{n}-m_{j}^{n}}{m_{l}^{n}+1}D_{m_{j}^{n}}(\theta-A_{l}^{n})
\end{flalign*}
which is what we wanted to show.
\end{proof}

\begin{lemma}
\label{lemma:upperBoundPartialFourierSumFn}
For all $\theta$ in the union of the $J_{i}^{n}$, $i=1,\dotsc,x_{n}-1$, and for all $j=1,\dotsc,x_{n}-1$, we have $$0 \le S_{m_{j}^{n}}(f_{n},\theta)-\frac{1}{x_{n}}\sum_{l=j+1}^{x_{n}}\frac{m_{l}^{n}-m_{j}^{n}}{m_{l}^{n}+1}D_{m_{j}^{n}}(\theta-A_{l}^{n}) \le C,$$ and $$0 \le S_{m_{j}^{n}}(f_{n},-\theta)-\frac{1}{x_{n}}\sum_{l=j+1}^{x_{n}}\frac{m_{l}^{n}-m_{j}^{n}}{m_{l}^{n}+1}D_{m_{j}^{n}}(\theta+A_{l}^{n}) \le C,$$ where $C$ has been defined in Lemma \ref{lemma:propertiesFejerKernel}.
\end{lemma}

\begin{proof}[Lemma \ref{lemma:upperBoundPartialFourierSumFn}]
Put $\theta$ in the union of the $J_{i}^{n}$, $i=1,\dotsc,x_{n}-1$, and $j \in \left\{1,\dotsc,x_{n}-1\right\}$. Lemma \ref{lemma:equalityPartialFourierSumFn} tells us that
\begin{equation}
\label{eq:equalityPartialFourierSumFn}
S_{m_{j}^{n}}(f_{n},\theta)-\frac{1}{x_{n}}\sum_{l=j+1}^{x_{n}}\frac{m_{l}^{n}-m_{j}^{n}}{m_{l}^{n}+1}D_{m_{j}^{n}}(\theta-A_{l}^{n})=\frac{1}{x_{n}}\sum_{l=1}^{j}F_{m_{l}^{n}}(\theta-A_{l}^{n})+\frac{1}{x_{n}}\sum_{l=j+1}^{x_{n}}\frac{m_{j}^{n}+1}{m_{l}^{n}+1}F_{m_{j}^{n}}(\theta-A_{l}^{n}).
\end{equation}
The Fejér kernel being non-negative, the left-hand side of equality (\ref{eq:equalityPartialFourierSumFn}) is non-negative as well.

In addition, Lemmas \ref{lemma:FejerKernelsUniformlyBounded} and \ref{lemma:equalityPartialFourierSumFn} immediately yield $$S_{m_{j}^{n}}(f_{n},\theta) \le \frac{1}{x_{n}}\sum_{l=1}^{j}C+\frac{1}{x_{n}}\sum_{l=j+1}^{x_{n}}\frac{m_{j}^{n}+1}{m_{l}^{n}+1}C+\frac{1}{x_{n}}\sum_{l=j+1}^{x_{n}}\frac{m_{l}^{n}-m_{j}^{n}}{m_{l}^{n}+1}D_{m_{j}^{n}}(\theta-A_{l}^{n}).$$ Since for all $l=j+1,\dotsc,x_{n}$, we have $(m_{j}^{n}+1) / (m_{l}^{n}+1) < 1$ because $m_{l}^{n} > m_{j}^{n}$, it follows that
\begin{equation}
\label{eq:upperBoundPartialFourierSumFn}
S_{m_{j}^{n}}(f_{n},\theta) \le C\frac{j}{x_{n}}+C\frac{x_{n}-j}{x_{n}}+\frac{1}{x_{n}}\sum_{l=j+1}^{x_{n}}\frac{m_{l}^{n}-m_{j}^{n}}{m_{l}^{n}+1}D_{m_{j}^{n}}(\theta-A_{l}^{n}).
\end{equation}
Inequality (\ref{eq:upperBoundPartialFourierSumFn}), together with the fact that the left-hand side of equality (\ref{eq:equalityPartialFourierSumFn}) is non-negative, yields the desired result for $S_{m_{j}^{n}}(f_{n},\theta)$. We use the exact same reasoning for $S_{m_{j}^{n}}(f_{n},-\theta)$.
\end{proof}

\subsection{Proof of Theorem \ref{thm:existenceEvenFunctionFourierDiverges}}
\label{subsection:appendixEvenFourier}

\begin{proof}[Theorem \ref{thm:existenceEvenFunctionFourierDiverges}]
This proof follows the proof of Theorem 3.1 in Chapter VIII, \S{3} of \cite{zygmund2003}. We present it here for completeness, giving particular attention to the finer details.

We consider the sequences $(\varphi_{n})_{n \in \mathbb{N}}$, $(M_{n})_{n \in \mathbb{N}}$, and $(E_{n})_{n \in \mathbb{N}}$, provided by Lemma \ref{lemma:HardyTrigonometricPolynomials}. For all $n \in \mathbb{N}$, we denote $q_{n}$ the degree of the polynomial $\varphi_{n}$. From Lemma \ref{lemma:HardyTrigonometricPolynomials}, we have $q_{n} \ge 1$ for all $n \in \mathbb{N}$.

Since the sequence $(M_{n})_{n \in \mathbb{N}}$ diverges to infinity, we can find by induction an increasing sequence $(n_{k})_{k \ge 1}$ of integers such that for all $k \ge 1$, we have $$M_{n_{k}}^{-\frac{1}{2}} < \frac{1}{2^{k}}.$$ It follows that $$\sum_{k=1}^{\infty}M_{n_{k}}^{-\frac{1}{2}} < +\infty.$$ The polynomial function $\varphi_{n_{k}}-\frac{1}{2}$ has constant term 0. Since it is even with degree $q_{n_{k}}$, it has the form $$\varphi_{n_{k}}(\theta)-\frac{1}{2}=\sum_{l=1}^{q_{n_{k}}}a_{l}^{(n_{k})}\left(\e^{\mathrm{i}l\theta}+\e^{-\mathrm{i}l\theta}\right)\hspace{-1.7pt},$$ where $a_{1}^{(n_{k})},\dotsc,a_{q_{n_{k}}}^{(n_{k})}$ are the coefficients associated with the monomials of degree $1,\dotsc,q_{n_{k}}$, respectively (because $\varphi_{n_{k}}-\frac{1}{2}$ is even, they are also the coefficients associated with the monomials of degree $-1,\dotsc,-q_{n_{k}}$, respectively). Since $\varphi_{n_{k}}-\frac{1}{2}$ is of degree $q_{n_{k}}$, the coefficient $a_{q_{n_{k}}}^{(n_{k})}$ is different from zero.

By induction, define a sequence of odd integers $(c_{k})_{k \ge 1}$ such that for all $k \ge 1$, $q_{n_{k}}c_{k} < c_{k+1}$. We have $$\varphi_{n_{k}}(c_{k}\theta)-\frac{1}{2}=\sum_{l=1}^{q_{n_{k}}}a_{l}^{(n_{k})}\left(\e^{\mathrm{i}lc_{k}\theta}+\e^{-\mathrm{i}lc_{k}\theta}\right)\hspace{-1.7pt},\,\text{and}~\varphi_{n_{k+1}}(c_{k+1}\theta)-\frac{1}{2}=\sum_{l=1}^{q_{n_{k+1}}}a_{l}^{(n_{k+1})}\left(\e^{\mathrm{i}lc_{k+1}\theta}+\e^{-\mathrm{i}lc_{k+1}\theta}\right)\hspace{-1.7pt}.$$ Thus, the polynomial $\varphi_{n_{k}}(c_{k}\theta)-\frac{1}{2}$ has degree $q_{n_{k}}c_{k}$, while the non-zero monomial term with lowest degree of the polynomial $\varphi_{n_{k+1}}(c_{k+1}\theta)-\frac{1}{2}$ has degree at least $c_{k+1}$. Since $q_{n_{k}}c_{k} < c_{k+1}$, it follows that the degrees of the non-zero monomial terms of the polynomial $\varphi_{n_{k}}(c_{k}\theta)-\frac{1}{2}$ are all smaller than those of $\varphi_{n_{k+1}}(c_{k+1}\theta)-\frac{1}{2}$. In particular, the degrees of the polynomials $\varphi_{n_{k}}(c_{k}\theta)-\frac{1}{2}$, $k \ge 1$, do not overlap.

We then write
\begin{equation}
\label{eq:definitionFunctionfFourierDivergent}
f(\theta)=\sum_{k=1}^{\infty}\frac{\varphi_{n_{k}}(c_{k}\theta)}{M_{n_{k}}^{\frac{1}{2}}}.
\end{equation}
The function $f$ is well-defined from $\mathbb{R}$ to $[0,+\infty]$, since for all $\theta \in \mathbb{R}$, $f(\theta)$ can be written as an infinite sum of non-negative terms. Moreover, it is clear that $f$ is non-negative and even. In addition, the monotone convergence theorem yields $$\int_{-\pi}^{\pi}f(\theta)d\theta=\sum_{k=1}^{\infty}\frac{1}{M_{n_{k}}^{\frac{1}{2}}}\int_{-\pi}^{\pi}\varphi_{n_{k}}(c_{k}\theta)d\theta=\sum_{k=1}^{\infty}\frac{1}{M_{n_{k}}^{\frac{1}{2}}}\frac{1}{c_{k}}\int_{-c_{k}\pi}^{c_{k}\pi}\varphi_{n_{k}}(\theta)d\theta.$$ Since for all $k \ge 1$, the function $\varphi_{n_{k}}$ is $2\pi$-periodic, it follows that $$\int_{-\pi}^{\pi}f(\theta)d\theta=\sum_{k=1}^{\infty}\frac{1}{M_{n_{k}}^{\frac{1}{2}}}\frac{1}{c_{k}}\sum_{l=0}^{c_{k}-1}\int_{-c_{k}\pi+2l\pi}^{-c_{k}\pi+2(l+1)\pi}\varphi_{n_{k}}(\theta)d\theta=\sum_{k=1}^{\infty}\frac{1}{M_{n_{k}}^{\frac{1}{2}}}\int_{-\pi}^{\pi}\varphi_{n_{k}}(\theta)d\theta=\pi\sum_{k=1}^{\infty}\frac{1}{M_{n_{k}}^{\frac{1}{2}}} < +\infty.$$ Thus, $f$ is integrable on $\left[-\pi,\pi\right]$. In particular, it is finite almost everywhere, and the series in the right-hand side of (\ref{eq:definitionFunctionfFourierDivergent}) is (absolutely) convergent for $m$-almost all $\theta \in \left[-\pi,\pi\right]$. In addition, we have $$f(\theta)=\frac{1}{2}\sum_{k=1}^{\infty}\frac{1}{M_{n_{k}}^{\frac{1}{2}}}+\sum_{k=1}^{\infty}\frac{\varphi_{n_{k}}(c_{k}\theta)-\frac{1}{2}}{M_{n_{k}}^{\frac{1}{2}}}=\sum_{k=0}^{\infty}\psi_{k}(\theta),$$ where $$\psi_{k}(\theta)=
\begin{cases}
\vspace{5pt}
\displaystyle\frac{1}{2}\sum_{l=1}^{\infty}\frac{1}{M_{n_{l}}^{\frac{1}{2}}} & \text{if}~k=0, \\
\displaystyle\frac{\varphi_{n_{k}}(c_{k}\theta)-\frac{1}{2}}{M_{n_{k}}^{\frac{1}{2}}} & \text{if}~k \ge 1.
\end{cases}$$ Take $n \in \mathbb{Z}$. Notice that for all $K \in \mathbb{N}$, we have $$\left|\sum_{k=0}^{K}\psi_{k}(\theta)\e^{-\mathrm{i}n\theta}\right| \le \sum_{k=0}^{K}\left|\psi_{k}(\theta)\right| \le \frac{1}{2}\sum_{l=1}^{\infty}\frac{1}{M_{n_{l}}^{\frac{1}{2}}}+\sum_{k=1}^{K}\frac{\varphi_{n_{k}}(c_{k}\theta)+\frac{1}{2}}{M_{n_{k}}^{\frac{1}{2}}} \le \sum_{k=1}^{\infty}\frac{\varphi_{n_{k}}(c_{k}\theta)+1}{M_{n_{k}}^{\frac{1}{2}}}.$$ Denote $g$ the function defined on $[-\pi,\pi]$ such that for all $\theta \in [-\pi,\pi]$, $$g(\theta)=\sum_{k=1}^{\infty}\frac{\varphi_{n_{k}}(c_{k}\theta)+1}{M_{n_{k}}^{\frac{1}{2}}}.$$ We have $g$ non-negative, and $g$ is integrable by the monotone convergence theorem. Thus, the sequence $$\left(\theta \mapsto \sum_{k=0}^{K}\psi_{k}(\theta)\e^{-\mathrm{i}n\theta}\right)_{K \ge 0}$$ converges pointwise to the function $\theta \mapsto f(\theta)\e^{-\mathrm{i}n\theta}$, and it is dominated by the integrable function $g$. Hence, the dominated convergence theorem guarantees that the series $\sum_{k \in \mathbb{N}}\hat{\psi}_{k}(n)$ converges and that
\begin{equation}
\label{eq:FourierCoefOffAsSumOfFourierCoefs}
\sum_{k=0}^{\infty}\hat{\psi}_{k}(n)=\lim_{K \to +\infty}\frac{1}{2\pi}\int_{-\pi}^{\pi}\sum_{k=0}^{K}\psi_{k}(\theta)\e^{-\mathrm{i}n\theta}d\theta=\frac{1}{2\pi}\int_{-\pi}^{\pi}f(\theta)\e^{-\mathrm{i}n\theta}d\theta=\hat{f}(n).
\end{equation}
But the family $(\psi_{k})_{k \in \mathbb{N}}$ is a family of non-overlapping polynomials. Thus, for all $n \in \mathbb{Z}$, there is at most one polynomial among all the $(\psi_{k})_{k \in \mathbb{N}}$ whose $n$-th Fourier coefficient is non-zero. Combined with equality (\ref{eq:FourierCoefOffAsSumOfFourierCoefs}), this means that for all $n \in \mathbb{Z}$, either $\hat{f}(n)=0$ and $\hat{\psi}_{k}(n)=0$ for all $k \in \mathbb{N}$, either $\hat{f}(n) \ne 0$ and there is exactly one $k \in \mathbb{N}$ such that $\hat{\psi}_{k}(n) \ne 0$, and we have $\hat{f}(n)=\hat{\psi}_{k}(n)$. Since the functions $\psi_{k}$, $k \in \mathbb{N}$, are trigonometric polynomials, their Fourier coefficients coincide with their polynomial coefficients. It follows that the non-zero Fourier coefficients of $f$ are exactly the non-zero coefficients of the polynomials $\psi_{k}$, $k \in \mathbb{N}$, ordered so as to make their degrees match.

In addition, the non-overlapping is sequential, in the sense that for all $k \ge 1$, the degrees of the non-zero monomial terms of the polynomial $\psi_{k-1}$ are all smaller than those of $\psi_{k}$, while the degrees of the non-zero monomial terms of the polynomial $\psi_{k+1}$ are all larger than those of $\psi_{k}$. Hence, the Fourier series of $f$ can be obtained by formally writing out in full the successive polynomials $\psi_{k}$, $k \in \mathbb{N}$.

For all $k \ge 1$, define the set $\mathcal{E}_{k}$ by $$\mathcal{E}_{k}=\left\{\theta \in [-\pi,\pi]~\middle|~c_{k}\theta \in \bigcup_{l=0}^{c_{k}-1}\left(E_{n_{k}}+(2l-c_{k}+1)\pi\right)\right\} \subset \left[-\pi,\pi\right].$$ In other words, we have $$\mathcal{E}_{k}=\frac{1}{c_{k}}\bigcup_{l=0}^{c_{k}-1}\left(E_{n_{k}}+(2l-c_{k}+1)\pi\right)\hspace{-1.7pt}.$$ But the sets $\left(E_{n_{k}}+(2l-c_{k}+1)\pi\right)$, $l=0,\dotsc,c_{k}-1$, are clearly pairwise $m$-almost disjoint. We obtain $$m(\mathcal{E}_{k})=\frac{1}{c_{k}}\sum_{l=0}^{c_{k}-1}m(E_{n_{k}}+(2l-c_{k}+1)\pi)=\frac{1}{c_{k}}\sum_{l=0}^{c_{k}-1}m(E_{n_{k}})=m(E_{n_{k}}).$$ Since $m(E_{n}) \rightarrow 2\pi$ as $n \to +\infty$, and $n_{k} \rightarrow +\infty$ as $k \to +\infty$, we deduce that $m(\mathcal{E}_{k}) \rightarrow 2\pi$ as $k \to +\infty$.

Now, put $$\mathcal{E}=\bigcap_{j=1}^{\infty}\bigcup_{k=j}^{\infty}\mathcal{E}_{k} \subset \left[-\pi,\pi\right].$$ For all $j \ge 1$, we have $$m(\mathcal{E}_{k}) \le m\hspace{-1.6pt}\left(\bigcup_{k=j}^{\infty}\mathcal{E}_{k}\right) \le 2\pi,$$ for any $k \ge j$. Since $m(\mathcal{E}_{k}) \rightarrow 2\pi$, it follows that $$m\hspace{-1.6pt}\left(\bigcup_{k=j}^{\infty}\mathcal{E}_{k}\right)=2\pi,$$ for all $j \ge 1$. Thus, we obtain that $m(\mathcal{E})=2\pi$.

It remains to show that the Fourier series of $f$ diverges unboundedly at every point of $E$.

Take $\theta \in \mathcal{E}$. Then, $\theta$ belongs to infinitely many $\mathcal{E}_{k}$. For all $j \ge 1$, denote $k_{j}$ the smallest integer larger than or equal to $j$ such that $\theta \in \mathcal{E}_{k_{j}}$. The sequence $(k_{j})_{j \ge 1}$ is non-decreasing and diverges to infinity.

For all $j \ge 1$, we have $\theta \in \mathcal{E}_{k_{j}}$, so that there exists $l \in \left\{0,\dotsc,c_{k_{j}}-1\right\}$ such that $c_{k_{j}}\theta-(2l-c_{k_{j}}+1)\pi$ belongs to $E_{n_{k_{j}}}$. Consequently, there is an integer $p_{k_{j}}(\theta)$ such that $\max(1,n_{k_{j}}) \le p_{k_{j}}(\theta) \le q_{n_{k_{j}}}$, and
\begin{equation}
\label{eq:lowerBoundPartialFourierSumPhinkjModulo}
\left|S_{p_{k_{j}}(\theta)}(\varphi_{n_{k_{j}}},c_{k_{j}}\theta-(2l-c_{k_{j}}+1)\pi)\right| \ge M_{n_{k_{j}}}.
\end{equation}
But the $p_{k_{j}}(\theta)$-th symmetric partial Fourier sum associated with $\varphi_{n_{k_{j}}}$ is $2\pi$-periodic, and since $c_{k_{j}}$ is odd, it follows from inequality (\ref{eq:lowerBoundPartialFourierSumPhinkjModulo}) that
\begin{equation}
\label{eq:lowerBoundPartialFourierSumPhinkjClean}
\left|S_{p_{k_{j}}(\theta)}(\varphi_{n_{k_{j}}},c_{k_{j}}\theta)\right| \ge M_{n_{k_{j}}}.
\end{equation}
Additionally, since $S_{p_{k_{j}}(\theta)}(\varphi_{n_{k_{j}}},c_{k_{j}}\theta)$ is equal to the polynomial $\varphi_{n_{k_{j}}}$ truncated at degree $p_{k_{j}}(\theta)$ and evaluated at $c_{k_{j}}\theta$, it is easy to see that $M_{n_{k_{j}}}^{-1/2}\hspace{-1.6pt}\left(S_{p_{k_{j}}(\theta)}(\varphi_{n_{k_{j}}},c_{k_{j}}\theta)-\frac{1}{2}\right)$ is equal to the polynomial $\psi_{k_{j}}$ truncated at degree $c_{k_{j}}p_{k_{j}}(\theta)$ and evaluated at $\theta$. As a result of the above discussion regarding the Fourier series of $f$, $M_{n_{k_{j}}}^{-1/2}\hspace{-1.6pt}\left(S_{p_{k_{j}}(\theta)}(\varphi_{n_{k_{j}}},c_{k_{j}}\theta)-\frac{1}{2}\right)$ is the sum of the successive monomial terms of the Fourier series of $f$, evaluated at $\theta$, whose degrees are not smaller than $c_{k_{j}}$ and not larger than $c_{k_{j}}p_{k_{j}}(\theta)$. Inequality (\ref{eq:lowerBoundPartialFourierSumPhinkjClean}) implies that this sum is larger than or equal to $M_{n_{k_{j}}}^{-1/2}\hspace{-1.6pt}\left(M_{n_{k_{j}}}-\frac{1}{2}\right)$ in absolute value.

Hence, for all $j \ge 1$, the sum of the successive monomial terms of the Fourier series of $f$, evaluated at $\theta$, whose degrees are not smaller than $c_{k_{j}}$ and not larger than $c_{k_{j}}p_{k_{j}}(\theta)$, is larger than or equal in absolute value to a quantity which diverges to infinity as $j$ goes to $+\infty$. Indeed, we have $n_{k_{j}} \rightarrow +\infty$ as $j$ goes to $+\infty$ because $(n_{k})_{k \ge 1}$ is an increasing sequence of integers and $(k_{j})_{j \ge 1}$ is non-decreasing and diverges to infinity. It follows that $\left(M_{n_{k_{j}}}^{-1/2}\hspace{-1.6pt}\left(M_{n_{k_{j}}}-\frac{1}{2}\right)\right)_{j \ge 1}$ diverges to infinity as $j$ goes to $+\infty$.

In informal words, you can find finite connected blocks of terms of the Fourier series of $f$, evaluated at $\theta$, whose sums grow larger and larger in absolute value, so as to approach $+\infty$ in absolute value.

This implies that the Fourier series of $f$, evaluated at $\theta$, diverges unboundedly. To see this, assume that there is a constant $C > 0$ such that for all $n \in \mathbb{N}$, $\left|S_{n}(f,\theta)\right| \le C$. From what has just been said, there is a finite connected block of terms of the Fourier series of $f$, evaluated at $\theta$, whose sum is larger than $2C$ in absolute value. Say that the minimal degree of this block is $m$ and its maximal~degree~is~$M$.

Since the function $f$ is even, we thus have $$\left|\sum_{l=m}^{M}\hat{f}(l)\left(\e^{\mathrm{i}l\theta}+\e^{-\mathrm{i}l\theta}\right)\right| > 2C.$$ If $m=0$, then we get $\left|S_{M}(f,\theta)\right| > 2C$ and we already have a contradiction. If $m \ge 1$, we have $$\left|\left|S_{m-1}(f,\theta)\right|-\left|\sum_{l=m}^{M}\hat{f}(l)\left(\e^{\mathrm{i}l\theta}+\e^{-\mathrm{i}l\theta}\right)\right|\right| \le \left|S_{m-1}(f,\theta)+\sum_{l=m}^{M}\hat{f}(l)\left(\e^{\mathrm{i}l\theta}+\e^{-\mathrm{i}l\theta}\right)\right|=\left|S_{M}(f,\theta)\right| \le C,$$ from which follows that $$\left|S_{m-1}(f,\theta)\right| \ge \left|\sum_{l=m}^{M}\hat{f}(l)\left(\e^{\mathrm{i}l\theta}+\e^{-\mathrm{i}l\theta}\right)\right|-C > C,$$ which contradicts the fact that for all $n \in \mathbb{N}$, $\left|S_{n}(f,\theta)\right| \le C$. Thus, the Fourier series of $f$, evaluated at $\theta$, diverges unboundedly. The proof is now complete, given that $\mathcal{E} \subset \left[-\pi,\pi\right]$ and $m(\mathcal{E})=2\pi$.
\end{proof}

\end{document}